\documentclass[11pt]{article}

\usepackage{academicons}
\usepackage{hyperref}
\usepackage{fancyhdr}
\usepackage{authblk}

\author{Collin Bleak, Luke Elliott, James Hyde}

\usepackage{thmtools,thm-restate}

\usepackage[utf8]{inputenc}
\usepackage{tikz}

\usepackage[a4paper,centering]{geometry}
\usepackage{mathrsfs,amsmath, tikz,verbatim, wasysym,xcolor, amsfonts, amssymb, amsthm, graphicx, hyperref, enumerate, xfrac}
\usepackage{color}

\newcommand{\image}{\operatorname{Image}}
\newcommand{\im}[1]{\image(#1)}

\newcommand{\homeo}{\operatorname{Homeo}}
\newcommand{\C}{\mathfrak{C}}
\newcommand{\autc}{\homeo(\C)}

\newcommand{\N}{\mathbb{N}}

\newcommand{\Z}{\mathbb{Z}}
\newcommand{\Cc}{K_\C}

\newcommand{\K}{\mathscr{K}}
\newcommand{\Kf}{\K^{\text{f.g.}}}
\newcommand{\Kk}[1]{\mathscr{K}_{#1}}
\newcommand{\Kkf}[1]{\K^\text{f.g.}_{#1}}
\newcommand{\Oo}[2]{[#1]_{#2}}
\newcommand{\Ooo}[1]{\mathfrak{X}_{#1}}

\newcommand{\Ss}[2]{\left\{ #1 \hspace{1mm} \middle | \hspace{1mm} #2 \right\}}
\newcommand{\Se}{\subseteq}

\newcommand{\Sne}{\subsetneq}
\newcommand{\Sn}{\subsetneq}
\newcommand{\Sm}{\setminus}

\newcommand{\An}[1]{\left\langle #1 \right\rangle}

\newcommand{\It}[1]{\operatorname{Sym}\left({#1}\right)}
\newcommand{\pstab}{\operatorname{pstab}}
\newcommand{\suppp}[1]{\operatorname{supt}\!\left(#1\right)}
\newcommand{\supt}[1]{\suppp{#1}}

\newcommand{\seteq}{:=}

\newcommand{\tooo}{\longrightarrow}
\newcommand{\fls}{flawless}
\newcommand{\vig}{vigorous}

\newtheorem{thm}{Theorem}[section]
\newtheorem{cor}[thm]{Corollary}
\newtheorem{lem}[thm]{Lemma}
\newtheorem{prop}[thm]{Proposition}

\theoremstyle{definition}

\newtheorem{defn}[thm]{Definition} 
\newtheorem{rmk}[thm]{Remark}
\newtheorem{que}[thm]{Question}

\newtheorem{exmp}[thm]{Example}

\newtheorem*{thm*}{Theorem}
\newtheorem*{defn*}{Definition}
\newtheorem*{cor*}{Corollary}

\title{Sufficient conditions for a group of homeomorphisms of the Cantor set to be two-generated
}

\begin{document}
\maketitle

{\flushleft \textit{Keywords: Full groups, Cantor space,  verbal subgroups, simple groups, finite generation.}\\
\vspace{.1 in}
\textit{MSC} (2010): Primary: 20F38, 37B05; Secondary: 20F65,
20E32}
\begin{abstract} We introduce and study two conditions on groups of homeomorphisms of Cantor space, namely the conditions of being vigorous and of being flawless.  These concepts are dynamical in nature, and allow us to study a certain interplay between the dynamics of an action and the algebraic properties of the acting group. %We study groups of homeomorphisms of a Cantor space $\C$ which are vigorous, or, which are flawless, where we introduce both of these terms here.  
A group $G\leq \autc$ is \emph{vigorous} if for any clopen set $A$ and proper clopen subsets $B$ and $C$ of $A$ there is $\gamma \in G$ in the pointwise-stabiliser of $\C\backslash A$ with $B\gamma\subseteq C$.  A non-trivial group $G\leq \autc$ is \emph{flawless} if for all $k$ and $w$ a non-trivial freely reduced product expression on $k$ variables (including inverse symbols), a particular subgroup $w(G)_\circ$ of the verbal subgroup $w(G)$ is the whole group.  We show: 1) {simple \vig{}} groups are either two-generated by torsion elements, or not finitely generated, 2) flawless groups are both perfect and lawless, 3) vigorous groups are simple if and only if they are flawless, and, 4) the class of vigorous simple subgroups of $\autc$ is fairly broad (the class is closed under various natural constructions and contains many well known groups such as the commutator subgroups of the Higman--Thompson groups $G_{n,r}$, the Brin-Thompson groups $nV$, R\"{o}ver's group $V(\Gamma)$, and others of Nekrashevych's `simple groups of dynamical origin').

\end{abstract}
\tableofcontents
\thispagestyle{empty}

\section{Introduction}

The papers \cite{Higman,Epstein, Ling84} provide conditions on a permutation group, or on a group of homeomorphisms, which imply simplicity of the commutator subgroup of that group.  Two of the core conditions in those works are that the group is generated by elements with ``small'' support (with differing definitions of the concept of small), and that the group acts nearly transitively (in the sense that given small sets can be sent into target small sets using elements which act as the identity off of some containing set which itself is ``small enough'').  Taken at a broad level of interpretation, one sees that these ideas also play a central role for the same result for large enough finite symmetric groups.

In this article we explore some consequences of a similar set of conditions on groups of homeomorphisms of Cantor space $\C$, by which we mean any space homeomorphic to $\{0,1\}^\omega$, the countably infinite product of the discrete space $\{0,1\}$ with itself.  The results here are new, but some of them bear comparison with results of Matui \cite{Matui06,Matui12,Matui13,Matui16} and of Nekrashevych \cite{Nekrashevych} for different (but related) families of groups.

A central new result is our Theorem \ref{thm:sub2}: for a large family $\Kf$ of finitely generated simple groups, each group in $\Kf$ is actually two-generated by torsion elements.  We are unaware of a similar result in the literature.  Note that this result applies to all of the known simple groups arising out of generalising the construction of the R. Thompson group $V$ (e.g., the commutator subgroups of The Higman--Thompson groups $G_{n,r}$ \cite{Higman}, the Brin-Thompson groups $nV$ \cite{Brin}, 
the R\"over group $V(\Gamma)$ \cite{Roever}, and many others of the `simple groups of dynamical origin' of Nekrasheych \cite{Nekrashevych}).  Specifically, we believe this theorem provides the first proof that R\"over's group $V(\Gamma)$ is $2$-generated.  Finally, Belk and Zaremsky in \cite{BelkZaermski2020} employ Theorem \ref{thm:sub2} as part of their argument which gives a quasi-isometric embedding of any finitely generated group into a two-generated simple group.

We also derive dynamical conditions guaranteeing simplicity, and our conditions determine that the groups involved satisfy no laws, and indeed, the stronger condition that they are mixed identity free.

Part of our motivation has been the well known open question as to whether there are any finitely presented simple groups which are not two generated (see \cite{GUBA} for finitely generated simple groups with every two-generated subgroup being a free group).  Thus, our results here exclude many natural candidate groups found within the setting of groups of homeomorphisms of Cantor space.  Somewhat artificially, we observe the well known fact (see Proposition \ref{prop:acige}) that any countably infinite group can be realised as a group of homeomorphisms of Cantor space. (Although, the easy proof we give of this produces groups which do not satisfy our dynamical condition.)

\subsection{Freedom of action}
In this paper we follow the notational convention occuring in many papers on groups of homeomorphisms of a space, and in particular, in many papers in the R. Thompson groups literature (see, e.g.,\cite{Brin}).  Specifically, standard operators usually act on the left (such as ``Homeo($\cdot$)'' and ``supt($\cdot$)'', below), but {\bf group actions and created functions in the text will act on the right}.  E.g., given an element $\gamma\in \autc$ we denote by $\supt{\gamma}$ the \emph{support of $\gamma$}, that is, the set of all points which are actually moved by $\gamma$: 
\[
\supt{\gamma}\seteq \{p\in \C\mid p\gamma\neq p\}.
\]
An example of this usage is that we also follow the convention established in the literature for verbal subgroups, where the sets $w[G]$ and $w(G)$ (defined below in Subsection \ref{subsec:wordLaw}) are naturally thought of as examples of ``left'' action notation.  Two standard consequences of using right action notation are the following formulae:
\begin{align*}
h^g &:= g^{-1}hg; \textrm{ and }\\
[g,h] &:= g^{-1}h^{-1}gh = (h^{-1})^{g} \cdot h = g^{-1} \cdot g^h.
\end{align*}

\begin{defn} \label{move}
We will say that a subset $S$ of $\autc$ is \emph{\vig{}} if and only if for all $A,B,C$ clopen subsets of $\C$ with $B$ and $C$ proper (non-empty) subsets of $A$ there exists $\gamma$ in $S$ with $\suppp{\gamma} \Se A$ and $B\gamma \Se C$.
\end{defn}
Note that in the above definition, the requirement that \(B, C\) be proper subsets of \(A\) cannot be removed. For example, in the case that \(C\subsetneq B=A = \C\), there is no \(\gamma\) with \(B\gamma \subseteq C\). Consequently we cannot replace \(A\) with \(B \cup C\).

We use the word \vig{} because it is evocative of thoroughly mixing the Cantor set.  The following lemma is immediate from the definition of vigorous.

\begin{lem} \label{Vcsg}
If $F$ and $G$ are subgroups of $\autc$ with $F$ \vig{} and contained in $G$ then $G$ is also \vig{}.
\end{lem}

The reader can check that R. Thompson's group $V_2$ is a \vig{} group (see Subsection \ref{subsec:HigThomp} for a definition of $V_2$).  Now by Lemma \ref{Vcsg}, we see that R\"over's group \cite{Roever} and the finitely generated simple overgroups of $V_2$ constructed in \cite{Nekrashevych} are vigorous as well.

\subsection{The word is law}\label{subsec:wordLaw}
Let  $F_j$ be a free group on symbols $x_1$, \ldots , $x_j$ and G be a group with $w = a_1a_2\dots a_m$ a word in the letters $\{x_1,... ,x_j\}^{\pm1}$.  Consider the function $w:G^j\to G$ deﬁned by mapping the element $(g_1,\ldots ,g_j)\in G^j$ to the element of $G$ obtained by substituting $g_i$ for every occurrence of $x_i$ in $w$. If we set $$w[G]\seteq\im{w}$$ and $$w(G)\seteq \langle w[G]\rangle,$$ then respectively the set $w[G]$ is called the \emph{verbal subset for $w$} and $w(G)$ is the \emph{verbal subgroup for $w$}.

It is immediate that verbal subgroups are characteristic, and a standard first example is given by taking $k = 2$ and $w = x_1^{-1}x_2^{-1}x_1x_2$, where $w(G)$ is then the commutator subgroup of $G$.  See \cite{MKS} for more on verbal subgroups.

We will denote the set of non-empty proper clopen subsets of $\C$ by $\Cc$ throughout.  Note that $A\in \Cc$ if and only if $\C\backslash A\in \Cc$.  If $G\leq \autc$ and $A\in \Cc$ then we use  $\pstab_G(A)$ to denote the elements in $G$ which pointwise-stabilise the set A. For $G\leq \autc$ we add a subscript $\circ$ to modify our verbal notations as follows:
    
   \[w[G]_\circ\seteq \bigcup_{A\in \Cc}w[\pstab_G(A)],\quad\quad w(G)_\circ\seteq \langle w[G]_\circ\rangle.\]

Recall that a group $G$ \emph{satisfies a law} if there is a non-trivial freely reduced word $w$ so that $w(\alpha_1,\alpha_2,\ldots,\alpha_j)=1_G$  for any choice of $(\alpha_1,\alpha_2,\ldots,\alpha_j)\in G^j$.  E.g, abelian groups satisfy a law since $w=x_1^{-1}x_2^{-1}x_1x_2$ evaluates to $1_G$ for all choices of values for the variables $x_1$ and $x_2$.  The law that is satisfied is typically stated as the equation in the variables $x_1$, $x_2$, $\ldots$, $x_j$ that is satisfied for all choices of substituting group elements. I.e., $w=1_G$.

If a group satisfies no laws we say it is \emph{lawless}.  Our first theorem is as follows.

\begin{restatable*}{thm}{vigLaw}\label{thm:vigLaw}
Let $G\leq \autc$. If $G$ is vigorous then $G$ is lawless.
\end{restatable*}

In fact, as pointed out by a helpful referee, a stronger theorem holds.

Let $G$ be a group, and $F$ a nontrivial free group. A mixed identity
for $G$ is $w\in G*F$ such that $1_G=(w)\phi$ for 
every homomorphism $\phi: G*F\to G$ such that 
$\phi|_G$ is the identity. A mixed identity $w$ is non-trivial if $w$ is not 
the identity element of $G*F$. The group $G$ is mixed identity free if 
it has no nontrivial mixed identity.  See \cite{Anashin, HullOsin} for more details.

\begin{restatable*}{thm}{vigMIF}\label{thm:vigMIF}
Let $G\leq \autc$. If $G$ is vigorous then $G$ is mixed identity free.
\end{restatable*}

\begin{defn} \label{gen}
A subgroup $G$ of $\autc$ is \emph{\fls{}} if and only if $G$ is equal to $w(G)_\circ$ for every non-trivial freely reduced word \(w\).
\end{defn}
While it might seem that (outside of the trivial subgroup) flawless groups ought to be hard to discover amongst the subgroups of $\autc$, in fact many well known groups of homeomorphisms of Cantor space are flawless.

We now state a lemma where each result is immediate from the definition of flawless.
\begin{restatable}{lem}{flsFacts} \label{lem:flsFacts}
Let $G\leq \autc$ be non-trivial and flawless.  Then
\begin{enumerate}
    \item $G$ is lawless,
    \item \label{lem:flsPerfect} $G$ is perfect,
    \item $G=w(G)$ for any non-trivial freely reduced word \(w\).
\end{enumerate}
\end{restatable}
Lemma \ref{lem:flsFacts} motivates our choice of the term \emph{flawless}.

To state our next result, we will need several more definitions.

If $\gamma \in \autc$ fixes a proper clopen set \(A\) pointwise then its support is contained in $\C\backslash A$. So the elements in $x[G]_\circ$  (that is, $w[G]_\circ$ with word $w=x$) are precisely those homeomorphisms which admit a proper clopen set containing their support: we call these the \emph{elements of small support}.

\textbf{}\begin{defn} \label{genSmall}
A subgroup $G$ of $\autc$ is \emph{generated by its elements of small support}  if and only if  $G=x(G)_\circ$.
\end{defn}

\subsection{(Approximately) full groups}
The next definition is used in \cite{Lawsonn}.

\begin{defn} \label{def:full} We say a subgroup $G$ of $\autc$ is \emph{full} if and only if for all $\gamma_1, \ldots, \gamma_n \in G$ and $D_1, D_2, \ldots, D_n$ clopen sets partitioning $\C$ so that  $D_1 \gamma_1, D_2\gamma_2,\ldots, D_n \gamma_n$ also partitions $\C$ we have the union of partial functions
\[\bigsqcup\limits_{1\leq i\leq n} {\gamma_i}|_{D_i}\]
is also in $G$.
\end{defn}
The following is a weaker condition which turns out to be natural in the context considered here.
\begin{defn} \label{def:afull}
We will say a subgroup $G$ of $\autc$ is \emph{approximately full} if and only if whenever $\gamma_1, \ldots, \gamma_n\in G$ and $D_1, \ldots, D_n$ are clopen sets partitioning $\C$ such that $D_1 \gamma_1, D_2\gamma_2, \ldots, D_n \gamma_n$ also partitions $\C$ and $j\in \{1, \ldots, n\}$ then there exists $\chi$ in $G$ such that $\chi$ extends
$\gamma_i|_{D_i}$ for each $i \in \{1, \ldots, n\} \Sm \{j\}$. 
\end{defn}
\begin{rmk}
For the curious reader, the commutator subgroup of the Higman--Thompson group $G_{3,1}$  is an example of a group of homeomorphisms of Cantor space which is approximately full but not full (the group $G_{3,1}$ is commonly denoted by $V_3$ following a notation introduced by Brown in \cite{BrownFinite}).

We give a sketch argument here. We say that an element \(f\in V_3\) is \textit{odd} if there is \(n\in \N\) such that whenever \(p_1< \ldots< p_{k}\in \{0, 1, 2\}^*\) (in dictionary order) are such that every element of \(\{0,1, 2\}^\N\) has precisely one of \(p_1, \ldots, p_{k}\) as a prefix and \(|p_1|, |p_2|, \ldots, |p_{k}|>n\), it follows that the permutation \(\phi:\{1, \ldots, k\} \to \{ 1, \ldots, k\} \) defined by
\[(p_{(1)\phi}00\ldots)f<(p_{(2)\phi}00\ldots)f<\ldots<(p_{(k)\phi}00\ldots)f\]
is odd. We define even analogously. It is routine to verify that every element of \(V_3\) is either even or odd but not both (here we use the fact that 3 is odd).
One can then verify that the commutator subgroup $D(V_3)$ of \(V_3\) consists of precisely the even elements of \(V_3\).
The full group of $D(V_3)$ is \(V_3\), but $D(V_3)$ is only approximately full (one can change an element from odd to even by modifying it on a small clopen set).

\end{rmk}

\begin{restatable*}{thm}{aes} \label{thm:aes}
Let $G$ be a subgroup of $\autc$. 
\begin{enumerate}
\item \label{aes:1}If $G$ is approximately full then $G$ is generated by its elements of small support.
\item \label{aes:2}If $G$ is \vig{} and generated by its elements of small support, then $G$ is approximately full.
\end{enumerate}
\end{restatable*}

In particular, we have the following corollary.
\begin{cor}\label{cor:aes}
Let $G$ be a \vig{} subgroup of $\autc$.  Then, $G$ is approximately full if and only if $G$ is generated by its elements of small support.
\end{cor}

Our next result is on simplicity.  Recall that a group is simple exactly if the normal closure of any non-identity element is the whole group. 

%{\color{orange}I think we should strengthen this result as follows}

\begin{restatable*}{thm}{fes} \label{thm:fes}
Let $G$ be a \vig{} subgroup of $\autc$. Then the following are equivalent
\begin{enumerate}
\item \label{fessim} $G$ is simple,
\item \label{fesfls} $G$ is \fls{},
\item \label{feslfl} $G=w(G)_\circ$ with \(w=x_1^{-1}x_2^{-1}x_1x_2\),
\item \label{fespss} $G$ is perfect and is generated by its elements of small support, 
\item \label{fespaf} $G$ is perfect and approximately full, and
\item \label{fesmat} $G$ is the commutator subgroup of its own full group within $\autc$.
\end{enumerate}
\end{restatable*}

The reader may recall that it is already known that point (\ref{fesmat}) implies point (\ref{fessim}), above,  when the full group of $G$ is the full group of an essentially principle minimal \'etale groupoid $\mathscr{G}$.  This is a theorem of Matui in \cite{Matui15}.  Indeed, Matui's proof of that direction essentially holds up in our context, although in general our context is quite different.  For example, our class of groups contains uncountable groups.  The remaining parts of the theorem are new. 

\subsection{On finitely generated vigorous simple groups}
We next turn our attention to the class of vigorous simple subgroups of $\autc$, and in particular, to the subclass of those groups which are finitely generated.

\begin{restatable*}{thm}{subbtwo} \label{sub2}\label{thm:sub2}
Let $G \leq \autc$ be finitely generated simple and \vig{}. Let $n\in\Z$ be at least $2$. Then there exists $\sigma$ and $\zeta$ in $G$ such that $\sigma$ is of finite order and $\zeta$ is of order $n$ and $\An{\sigma,\zeta} = G$.
\end{restatable*}

\begin{rmk}
    It follows that the R\"over group $V(\Gamma)$ \cite{Roever} is $2$-generated.  R\"over actually builds an infinite family of finitely presented simple groups.  As  R\"over's constructed groups are all overgroups of R. Thompson's group $V_2$ in $\autc$, they are all vigorous.  Thus, by Theorem \ref{thm:sub2}, they are all $2$-generated by torsion elements.
\end{rmk}
\begin{defn} \label{defK}
We will write $\K$ for the family of subgroups of $\autc$ which are simple and \vig{}. We will write $\Kf$ for the family of finitely generated groups in $\K$.
\end{defn}

\subsection{An outline of what follows}
The remainder of the paper can be loosely described as follows:

In Section \ref{sec:Cantor} we give some background information on Cantor space and some well known groups of homeomorphisms of Cantor space which happen to be vigorous (i.e., the Higman--Thompson groups).

In Section \ref{sec:idInVig} we explore properties of vigorous groups: namely, that they are lawless and indeed satisfy the stronger property of being mixed identity free.  We also provide an example group that is vigorous, perfect, and lawless, but not simple (nor flawless).

In Section \ref{sec:equivConditions} we prove Theorems \ref{thm:aes} and \ref{thm:fes}.  This section has an initial part where we explore some dynamical properties of being vigorous, and two following subsections.

In Subsection \ref{subsec:full} we introduce the further concept of \emph{strongly approximately full}, and explore the relationships between the various versions of ``full'', with or without the extra property of being vigorous.  Here we prove Theorem \ref{thm:aes}.  We also touch on the 
interplay between a group being \emph{flexible} (a concept which appears in many versions of Rubin's spatial reconstruction theorem \cite{Rubin}) and being vigorous.  Specifically, we show that for a group of homeomorphisms of Cantor space, being flexible is a weaker property than being vigorous, but if the group is approximately full, then vigorous and flexible become equivalent properties.  

In the following Subsection \ref{subsec:equiv} we prove Theorem \ref{thm:fes}
 which we recall states that for a vigorous group, many different conditions are equivalent (one of these is simplicity).

In Section \ref{sec:exploreSimpleVig} we explore properties of the groups in the classes $\K$ and $\Kf$.  We show that $\K$ and $\Kf$ are closed under some natural constructions, and of course recall that $\Kf$ contains some well known groups such as all of the standard simple generalisations of R. Thompson's group $V_2$ (e.g., the commutator subgroups $D(V_{n})$ of the Higman-Thompson groups $V_{n}$ (note that $D(V_{n})=V_n$ when $n$ is even, and $D(V_{n})$ is an index two subgroup of $V_n$ when $n$ is odd), the R\"over group $V_\Gamma$, the Brin-Thompson groups $nV$, and the finitely generated simple groups of dynamical origin of Nekrashevych \cite{Nekrashevych}).  Given a simple vigorous group $G$ we introduce a constructed ``homology'' group $\Ooo{G}$ (following a model provided by Matuii which applies to our context), and we use properties of these induced groups to prove Theorem \ref{thm:sub2} that any groups in $\Kf$ are $2$-generated by torsion elements.

In Section \ref{sec:conc} we give a brief discussion of some directions for potential future work.

This paper contains parts of Chapter 3 of the dissertation \cite{HydeDiss}. The definitions of being \fls{}, \vig{} and approximately full are original to this work, though approximately full is similar to full as discussed in \cite{Matui12,Nekrashevych,Lawsonn}. As alluded above, we use the word \fls{} because of its association with both perfection and lawlessness.

We emphasise that unlike the properties of being simple or perfect the properties of being \vig{}, \fls{}, full and approximately full are dependent not just on the groups but also on their actions on $\C$.%  The following should be clear to the reader.

%BHP% consider adding in a Basic Definitions section here to include definitions of (strict?) cones and P_n and V_n.
\subsection{Acknowledgements}
We are grateful for various conversations on this topic with Jim Belk, Martin Kassabov, Mark Lawson, Francesco Matucci, Volodymyr Nekrashevych, Martyn Quick, and Nik Ru\v{s}kuc. We are also grateful for suggestions from our anonymous referees which have improved some of the results and much of the writing in this paper. The first author is also grateful for support from EPSRC grant EP/R032866/1.    

\section{On Cantor space}\label{sec:Cantor}

For natural $n>1$ we give the set $\{0,1,\ldots, n-1\}$ the discrete topology and set $\C_n\seteq \{0,1,\ldots,n-1\}^\omega$ to which we give the product topology. Specifically, we have 
\[
\C_n\seteq\{x_0x_1x_2\ldots\mid \forall i\in\N, x_i\in\{0,1,\ldots,n-1\}\}.
\]

For a given finite string $w=w_0w_1\ldots w_k$ (for some natural $k$) where each $w_i\in \{0,1,\ldots, n-1\}$, the \emph{cone at $w$} is the set
\[
w\C_{n}\seteq\Ss{w_0w_1\ldots w_kw_{k+1}w_{k+2}\ldots}{ \forall j>k, w_j \in\{0,1,\ldots,n-1\}}\subseteq \C_n.
\]

The standard basis for the topology of $\C_n$ is then the full set of such cones, which we might call the \emph{basic open cones of $\C_n$}.  

Central to our point of view is Brouwer's characterisation \cite{Brouwer} of Cantor space as a non-empty, compact Hausdorff space, without isolated points, and having a countable basis consisting of clopen sets. From this view, we see that all of the spaces $\C_n$ so defined are homeomorphic.

Now, for any given generic Cantor space $\C$ (i.e., those which are not amongst the spaces $\C_n$ defined above) we will assume an implicit homeomorphism $\phi$ to $\{0,1\}^\omega$, which then provides an explicit basis of open sets for the topology of $\C$.   For the topology of $\C_n$ we will use the basis as given above.

In the body of this paper, we will typically work with clopen sets, that is, sets which are both open and closed.  In $\C_n$ these sets are precisely those sets which can be written as unions of finitely many basic open cones.

Mostly relevant for subsets of Cantor space, but perhaps appearing elsewhere below, we use the notation \(\sqcup\) for unions of sets where the sets appearing in the union are guaranteed to be disjoint from each other (and particularly, only where it is easily clear from previous lines that the sets are disjoint).  That is, the notation is meant as a helpful reminder that the sets are disjoint, and not as a claim to that effect. 
In the cases that the sets appearing in the union may not be disjoint, or where it is not intended that the reader already knows that they are disjoint, we use the usual \(\cup\) notation.
\subsection{The Higman--Thompson groups $V_n$}\label{subsec:HigThomp}

A family of groups that arise often in what follows are the Higman--Thompson groups $G_{n,1}$ from Higman's book \cite{Higman}.  As alluded in the introduction we shall follow Brown's notation from \cite{BrownFinite} where he uses $V_n$ to represent the group $G_{n,1}$.  For a given natural $n>1$, the group $V_n$ is the group of homeomorphisms of $\C_n$ generated by the transpositions $(\alpha\;\;\beta)$, which are the homeomorphisms of $\C_n$ which transpose the points of $\C_n$ with finite prefix $\alpha$ with the corresponding points of $\C_n$ with finite prefix $\beta$, and otherwise act as the identity.  That is, if $\alpha=\alpha_0\alpha_1\ldots\alpha_k$, and $\beta=\beta_0\beta_1\ldots\beta_j$ are finite non-empty strings of symbols over the set $\{0,1,\ldots,n-1\}$ (which plays the role of an alphabet), such that neither $\alpha$ nor $\beta$ is a prefix of the other, then for all $\vec{z}\in\C_n$ we have
\[
\vec{z}\cdot(\alpha\;\;\beta)=
\left\{
\begin{array}{ll}
\vec{z} &\textrm{ if } \vec{z}\not\in \alpha\C_{n}\sqcup\beta\C_{n}\\
\beta_0\beta_1\ldots\beta_j z_{k+1}z_{k+2}\ldots & \textrm{ if } \vec{z} = \alpha_0\alpha_1\ldots\alpha_k z_{k+1}z_{k+2}\ldots\in \alpha\C_{n}\\
\alpha_0\alpha_1\ldots\alpha_k z_{j+1}z_{j+2}\ldots &\textrm{ if } \vec{z}=\beta_1\beta_1\ldots \beta_j z_{j+1}z_{j+2}\ldots\in\beta\C_{n}.
\end{array}
\right.
\] The group $V_n$ then consists of those homeomorphisms of $\C_n$ which can be thought of as finitary `prefix-exchange' maps.   The paper \cite{BCMNO} contains a corresponding definition of the groups $G_{n,r}$ using alphabets, strings, prefix maps, etc., while Higman in \cite{Higman} gave the initial definitions as groups of automorphisms of particular algebras.  We note in passing that $V_n$ is a dense subgroup of $\homeo(\C_n)$. 
 For each $n>1$ the group $V_n$ has a simple commutator subgroup $D(V_n)$, and this is the whole group if $n$ is even, or an index two subgroup if $n$ is odd.  The groups $V_n$ are finitely presented.  Thompson introduced $V\seteq V_2$ and a related group $T$ as the first examples of infinite, finitely presented simple groups in 1965 (see \cite{Thompson}).  See \cite{BleakQuick} for a presentation of $V_2$ using this generating set. 
 
We state the following result at the full level of generality of the  Higman--Thompson groups $G_{n,r}$, though we shall only prove it for the groups $V_n$ (the informed reader will see that the proof easily generalises to the broader context).
\begin{lem}\label{lem:vigHig}
The Higman--Thompson groups $G_{n,r}$ are vigorous. 
\end{lem}
\begin{proof}
Let $n>1$ be an integer and consider $V_n$ acting on $\C_{n}$.

Let $A$, $B$, $C$ be non-empty clopen sets in $\C_{n}$ so that $B$ and $C$ are proper subsets of $A$.

Let $u_1, u_2,\ldots,u_k\in\{0,1,\ldots,n-1\}^*$ be so that $B=u_1\C_n\sqcup u_2\C_n\sqcup\cdots\sqcup u_k\C_n$ (that is, these cones are pairwise disjoint 
and form a partition of $B$: note that such a decomposition exists as $B$ is compact and open, and the
cones of $\C_n$ form a basis for its topology).  Further, let $v_1,v_2,v_3\in\{0,1,\ldots,n-1\}^*$ so that 
$v_1\C_n\cap v_2\C_n=\emptyset$, 
$v_1\C_n\subseteq A\backslash B$, $v_2\C_n\subset A\backslash C$ and 
$v_3\C_n\subseteq C$. Define the elements $\gamma_1, \gamma_2, \gamma_3$ in $V_n$ as follows:
\begin{align*}
\gamma_1&=(u_1\;\;v_1u_1)(u_2\;\;v_1u_2)\cdots (u_k\;\;v_1u_k)\\
\gamma_2&=(v_1u_1\;\;v_2u_1)(v_1u_2\;\;v_2u_2)\cdots (v_1u_k\;\;v_2u_k)\\
\gamma_3&=(v_2u_1\;\;v_3u_1)(v_2u_2\;\;v_3u_2)\cdots (v_2u_k\;\;v_3u_k)
\end{align*}
It is now follows that the composition $\gamma_1\gamma_2\gamma_3$ moves $B$ into $C$, with the action fully supported in $A$.

Note that we built this element as a product of sets of (disjoint support) transpositions in three stages, as it could be the case that $C\subset B$, for instance, and so we cannot just transpose any cone in $B$ to one in $C$ without possibly creating an ill-defined map (transpositions are only defined for a pair of disjoint cones).

    \end{proof}

Unless we are discussing some explicit example group (such as within Example \ref{VFP}), from this point forward, we will generally operate in the setting of a Cantor space $\C$, without needing to explicitly reference the details of the construction of the clopen subsets of $\C$.

\section{Exploring identities for vigorous groups}\label{sec:idInVig}%%%%%%%%%%%%%%%%%%%%%%%%%%%%%%%%%%%%%%%%%

In this section we explore some properties of vigorous groups.  In particular, we show that they admit non-abelian free subgroups and hence are lawless.  We go further and show that they have the much stronger property of being mixed identity free, and also that the property of being vigorous is inherited by overgroups.  We close by giving an example of a perfect vigorous group which is not simple (and in particular, is not flawless).

Recall that we denote the set of non-empty proper clopen subsets of $\C$ by $\Cc$ throughout. 
\begin{lem}\label{lem:niceorbit}
Let $G\leq \homeo(\C)$ be vigorous and let $\{A,B,C\}\subset \Cc$ be a partition of $\C$ into three disjoint proper clopen subsets.  Then there is $\gamma\in \pstab_G(A)$ so that for all $n\in\Z\backslash \{0\}$ we have $B\gamma^n\subseteq C$.
\end{lem}
\begin{proof}
Let $A$, $B$ and $C$ be in $\Cc$ be such that $\{A,B,C\}$ is a partition of $\C$. 
Let \(D, E\in \Cc\) be such that \(\{D, E\}\) is a partition of \(C\).
Let $\gamma \in \pstab(A)$ be such that $(B \sqcup E)\gamma \Se E$. The element \(\gamma\) is as required.
\end{proof}

The following is a version of the Ping Pong Lemma similar to that appearing in   \cite{PDLH} as the Table-Tennis Lemma.  

\begin{lem}[Ping Pong Lemma]
Let $G$ be a group acting on a set $X$, let $\Gamma_1$, $\Gamma_2$ be two subgroups of $G$.  Let $\Gamma$ be the subgroup of $G$ generated by $\Gamma_1$ and $\Gamma_2$;  assume that $\Gamma_1$ contains at least three elements and $\Gamma_2$ contains at least two elements.  Assume that there exists two non-empty subsets $X_1$, $X_2$ in $X$, with $X_2$ not included in $X_1$, such that 
\begin{align*}
    (X_2)\gamma\subseteq X_1 \textrm{ for all } \gamma\in\Gamma_1\backslash \{1_\Gamma\},\\
    (X_1)\gamma\subseteq X_2 \textrm{ for all } \gamma\in\Gamma_2\backslash \{1_\Gamma\}.
\end{align*}
Then $\Gamma$ is isomorphic to the free product $\Gamma_1*\Gamma_2$.
\end{lem}
We will now prove vigorous groups have non-abelian free subgroups through an application of the Ping-Pong Lemma.

\begin{thm}\label{thm:vigfree}
Let $G \leq \autc$ be a \vig{} group and let $A$ be in $\Cc$. Then a free group of rank $2$ embeds in $\pstab_G(A)$.
\end{thm}

\begin{proof}

Let $B$, $C$ be in $\Cc$ such that $\{A,B,C\}$ is a partition of $\C$. 

By Lemma~\ref{lem:niceorbit}, let $\gamma \in \pstab(A)$ be such that for all \(n\neq 0\) we have $(C)\gamma^n \subseteq B$.
By Lemma~\ref{lem:niceorbit} again, let $\tau \in \pstab(A)$ be such that for all \(n\neq 0\) we have $(B)\tau^n \subseteq C$ .

  The last two sentences allow us to apply the Ping-Pong Lemma with $X:= B \sqcup C $, $\Gamma_1:=\An{\gamma}$, $\Gamma_2:=\An{\tau}$,  $X_1:=B$, and $X_2:=C$, to conclude that $\An{\tau, \gamma}$ is a free group of rank $2$.
\end{proof}

Theorem \ref{thm:vigfree} above gives the following immediate corollary as the free group of rank $2$ is lawless and lawlessness is inherited by overgroups.

\vigLaw

 We now prove an extension provided by a helpful anonymous referee, that if $G$ is vigorous, then it is mixed-identity free.

\vigMIF

\begin{proof}
     
 Note that by Remark 5.1 of \cite{HullOsin}, we only need to show the case where $F=\Z$ in the general definition of mixed identity free.
Let \(w(x) \in (G * \langle x\rangle) \backslash \{1\}\), we must find \text{\(\gamma\)} \text{\(\in\)} G such that \(w(\text{\(\gamma\)}) \neq 1\). By conjugating \(w\), we may assume
\(w = x^{n_1} g_1x^{n_2} g_2 \dots x^{n_k} g_k\)
where all the \(n_i \in \mathbb{Z} \backslash\{0\}\) and \(g_i\neq 1\).
We then find pairwise disjoint proper clopen sets \(U_i, V_i\) such that \(U_ig_i = V_i\) with \(X \backslash \text{\(\sqcup\)}_{i=1}^{k} (U_i \text{\(\sqcup\)}
V_i)\neq \varnothing\). 
Let \(V_0\) be a nonempty clopen subset of \(X \backslash \text{\(\sqcup\)}_{i=1}^{k} (U_i \text{\(\sqcup\)} V_i)\). 
Using Lemma \ref{lem:niceorbit} we find for \(i = 1, ..., k\) an element \(\text{\(\gamma\)}_i \text{\(\in\)} G \)
supported on \(V_{i-1} \text{\(\sqcup\)} U_i\) such
that \(V_{i-1}\text{\(\gamma_i\)}^{n_i} \text{\(\subseteq\)} U_i\). By construction the elements \(\gamma_1, \ldots, \text{\(\gamma\)}_k\) have disjoint supports. Let \(g=\gamma_1 \ldots \text{\(\gamma\)}_k\). It is then
straightforward to check that \((V_0)(w(g)) \subseteq V_k\), in particular \(w(g)\neq  1\) as wanted.
\end{proof} 

\begin{comment}
\begin{proof}
If $G$ is a \fls{} subgroup of $\autc$ then the set
\[\Ss{[g,h]}{g, h \in \pstab_G(A) \text{ for some non-empty clopen set } A}\]
generates $G$. Since this set is contained in the set of commutators of $G$ they must also generate $G$ as desired.
\end{proof}
\end{comment}

Noting that the verbal subgroups of non-abelian free groups are not typically the whole group, and that we can realise non-abelian free groups in $\autc$ (and indeed any countable group, see Corollary~\ref{cor:countableembeds}), we see that there are lawless subgroups of $\autc$ which are not flawless.  

The example below is of a perfect and vigorous subgroup of $\autc$ which is therefore lawless but not \fls{}. 

\begin{exmp} \label{VFP}
If $\gamma\in A_6$ and $l_1l_2l_3\dots$ is in $\{0,1,2,3,4,5\}^\omega$ then set 
\[(l_1l_2l_3\dots)\dot{\gamma} := (l_1\gamma)(l_2\gamma)(l_3\gamma)\dots\] 
and further set $\dot{A}_6$ as the copy of the alternating group $A_6$ in the group\\ $\homeo({\{0,1,2,3,4,5\}^\omega})$ generated by these functions $\dot{\gamma}$.

As $V_6$ is vigorous by Lemma \ref{lem:vigHig} we have by Lemma \ref{Vcsg} that the group $\An{V_6 \cup \dot{A}_6}$ is also \vig{}. The group $\An{V_6 \cup \dot{A}_6}$ is perfect since both $V_6$ and $\dot{A}_6$ are perfect. The group $\An{V_6 \cup \dot{A}_6}$ is not simple because $V_6$ is a normal subgroup. The group $\An{V_6 \cup \dot{A}_6}$ is therefore not \fls{} by Theorem \ref{thm:fes}. 
\end{exmp}

\section{On being vigorous}\label{sec:equivConditions}
In this section we prove Theorems \ref{thm:aes} and \ref{thm:fes}.
We begin by exploring some initial consequences of being vigorous.

\begin{lem} \label{FiN}
If $G$ is a subgroup of $\autc$ and $w$ is a word, then $w(G)_\circ \unlhd G$.
\end{lem}
\begin{proof}
Let $w$ be a word in the letters \(x_1, x_2, \ldots , x_j\). If \(g_1,g_2, \ldots g_j , \in G\) are supported on \(A\in \Cc\) with \(w(g_1,g_2, \ldots , g_j)\neq 1_G\) and \(h\in G\), then \(g_1^h,g_2^h, \ldots , g_j^h\) are supported on \(Ah\) and 
\[w(g_1^h,g_2^h, \ldots , g_j^h)=(w(g_1,g_2, \ldots , g_j))^h.\]
In particular $w[G]_\circ$ is closed under conjugation so the group it generates is normal.
\end{proof}
\begin{lem} \label{ntms}
Let $\gamma\in \autc$ be non-trivial.  There exists a non-empty clopen subset $Y$ of $\C$ with $Y\gamma$ disjoint from $Y$.
\end{lem}
\begin{proof}
%One can simply take a small enough proper clopen set around a point which is moved.

Let \(x\in \supt{\gamma}\). As \(\supt{\gamma}\) is open, there are non-empty disjoint clopen subsets \(A, B\) of \(\supt{\gamma}\) containing \(x\) and \((x)\gamma\) respectively. Let \(A'= A\cap (B)\gamma^{-1}\). Note \(A'\) is a clopen neighbourhood of \(x\). Then \((A')\gamma \subseteq B\) which is disjoint from \(A\) (and hence \(A'\)).
\end{proof}

%%%%%%%%%%%%%%%%%%%%%%%%%%%%%%%%%%%%%%%%%%%
\begin{lem}\label{lem:ssgp}
Let $G$ be a \vig{} subgroup of $\autc$. Let $U$ be in $\Cc$. Then
$$S_{G,U} \seteq \{\gamma \in G \mid \exists V\textrm{ a proper clopen subset of }U \textrm{ with }\supt{\gamma}\subseteq V\}$$
generates $\pstab_G(\mathfrak{C}\backslash U)$.
\end{lem}
\begin{proof}
Let $\eta$ in $G$ be wholly supported on $U.$ We will find $\mu$ and $\nu$ in $S_{G,U}$ such that $\mu \nu = \eta.$ Let $P \in \Cc$ be such that $P\eta^{-1} \cup P$ is a proper subset of $U.$ Let $Q \in \Cc$ be a proper subset of $\C \backslash U$.

Since $G$ is vigorous we may find $\phi \in G$ with support in $P\sqcup (\C \backslash U)$ such that $(P \sqcup Q)\phi \subseteq P$. Set $\nu \seteq \eta^\phi$. It follows that $\nu$ is wholly supported on $U \backslash (Q\phi)$, a proper clopen subset of $U$.

The homeomorphisms $\nu$ and $\eta$ must agree off $\supt{\phi}\eta^{-1} \cup \supt{\phi}$.  Note that  $\supt{\phi}\eta^{-1} \cup \supt{\phi} \Se (P\eta^{-1} \cup P)\sqcup (\C \Sm U)$. It follows that $\mu \seteq \eta\nu^{-1}$ is wholly supported on $(P\eta^{-1} \cup P)\sqcup (\C \Sm U)$.  However, $\supt{\eta}\subseteq{U}$ and $\supt{\nu} \subseteq U \Sm (Q\phi) \subseteq U$.  Therefore, $\supt{\mu}=\supt{\eta\nu^{-1}}\subseteq (P\eta^{-1} \cup P)$, a proper clopen subset of $U$.
\end{proof}

\begin{lem}\label{lem:SGDisjointGen}
Let $G\leq \autc$ be a vigorous group generated by its elements of small support.  Suppose there are disjoint sets $B,C,D\in\Cc$ partitioning $\C$.  Then $G=\An{\pstab_G(C)\cup\pstab_G(D)}$. 
\end{lem}
\begin{proof}

Let $L\in\Cc$ and suppose $\eta\in \pstab_G(L)$.  We will write $\eta$ as a product of elements of $\pstab_G(C)$ and $\pstab_G(D)$. 

 By the symmetry of context of $C$ and $D$, we may assume without meaningful loss of generality that $L\cap (B\sqcup C)$ is not empty (and is also proper and clopen). 

As $G$ is vigorous there is $\mu\in \pstab_G(D)$, wholly supported on $(B \sqcup C)$, with $C\mu^{-1}\subseteq L\cap(B\sqcup C)$.

Consider $\nu\seteq \eta^{\mu}$.  Observe that 
\[
\supt{\nu}= \supt{\eta}\mu\subseteq(\C\backslash{L})\mu=\C\backslash( L\mu)\subseteq \C\backslash C.
\]
In particular, we see that $\nu\in\pstab_G(C)$.  

We have found $\eta=\mu\nu\mu^{-1}$, a decomposition of $\eta$ as a product of elements from the sets $\pstab_G(C)$ and $\pstab_G(D)$, as desired.

\end{proof}

\begin{lem}\label{lem:ACD}
Let $G$ be a vigorous subgroup of $\autc$. Let $A$ be in $\Cc$ and let $C$ and $D$ be clopen subsets of $\C$ with $A$, $C$, and $D$ pairwise disjoint and with  $A\sqcup C\sqcup D\neq \C$.  Then, 
\[
\pstab_G(A)=\left\langle \pstab_G(A\sqcup C)\cup\pstab_G(A\sqcup D)\right\rangle.
\]
\end{lem}
\begin{proof}
Let \(B := \mathfrak{C} \backslash (A \sqcup C \sqcup D)\). Observe that the group \(\pstab_{G}(A)\) acts vigorously on the Cantor space \(\C\backslash A\). By Lemma \ref{lem:ssgp} we have that \(\pstab_{G}(A)\) is generated by its elements of small support (those with support contained in a proper and clopen subset of \(\C\backslash A\)). The result then follows from Lemma \ref{lem:SGDisjointGen}.
\end{proof}

The following lemma is useful in the proof of both of the theorems in this section.  In the statement below, note that even though the image of $\varepsilon$ is not all of $\C$, the function $f$ is still a well defined and injective homomorphism of groups.

\begin{lem} \label{lem:minime}
Let $G$ be a \vig{} subgroup of $\autc$ generated by its elements of small support.  Let $B$, $C$ and $D$ in $\Cc$ be such that $\{B,C,D\}$ is a partition of $\C$. Let $\delta \in \pstab_G(D)$ be such that $B\delta \Sn B$. Let $\varepsilon:\C \tooo \C$ be the injective function agreeing with $\delta$ on $B$ and pointwise stabilising $C \sqcup D$.

Let $f:\autc \tooo \autc$ be the function that takes $\gamma \in \autc$ to the element of $\autc$ which agrees with $\varepsilon^{-1}\gamma\varepsilon$ on $\C \varepsilon$ and fixes the rest of $\C$ pointwise. % Let $H$ be the group $\Ss{\gamma \in G}{\suppp{\gamma} \Se \C \varepsilon}$.
Then $f|_G$ is an isomorphism from $G$ to $\pstab_G(\C \Sm \C\varepsilon)$. 
\end{lem}
%[N.B., even though the image of $\varepsilon$ is not all of $\C$, the element $f$ is a well defined and injective homomorphism of groups.]
\begin{proof}
Set $H\seteq \pstab_G(\C \Sm \C\varepsilon)$ and
\[S_{G,\C\varepsilon}\seteq\Ss{\gamma \in G}{\exists V\textrm{ a proper clopen subset of }\C \varepsilon \textrm{ with} \suppp{\gamma} \subseteq V}.\]

One can think of $f$ as a restriction and co-restriction of conjugation by a homeomorphism $\hat{\varepsilon}$ which is an extension of $\varepsilon$ to a larger Cantor space. Consequently, $f$ is a group monomorphism.  It remains to show $(G)f = H$.

As $G$ is vigorous it follows from Lemma \ref{lem:ssgp} that $\An{S_{G,\C\varepsilon}} = \pstab_G(\C \Sm \C\varepsilon)=H$. 

Note that $B\varepsilon$ is a non-empty clopen set properly contained in $B$.  Set $A\seteq B\Sm (B\varepsilon)$, which is therefore another proper clopen subset of $B$.  Observe further that $ \C\varepsilon= B\varepsilon \sqcup C \sqcup D$, with complement in $\C$ being $A$.  In particular, $\pstab_G(A)=H$. 

Using Lemma \ref{lem:ACD}, where the sets $A$, $C$ and $D$ in this proof are used in the same way as their namesakes in Lemma \ref{lem:ACD}, we obtain that 
\[
H=\pstab_G(A)=\langle \pstab_G(A\cup C)\cup \pstab_G(A\cup D)\rangle.
\]

By Lemma \ref{lem:SGDisjointGen} we observe that $G=\An{\pstab_G(C)\cup\pstab_G(D)}$.  It follows that $(G)f=\An{\pstab_G(C)f\cup \pstab_G(D)f}$.

Rename $\delta$ as $\delta_D$  (as it stabilises $D$ pointwise).  As $G$ is vigorous there is $\theta$ with support contained in $C\sqcup D$ so that $C\theta^{-1}\subseteq D$.
Set $\delta_C\seteq\delta_D^\theta$, observing that $\delta_C|_B=\delta_D|_B=\varepsilon|_B$ and that $\delta_C$ has support contained in $B\sqcup D$ (and thus it stabilises $C$ pointwise).

We claim that for all \(\gamma\in \pstab_G(C)\) we have \((\gamma)f = \delta_D^{-1} \gamma\delta_D\). 
Indeed, for any \(p \in \mathfrak{C}\), 
\begin{enumerate}[-]
    \item If \(p\in A\), then \(p\delta_D^{-1}
 \in C\) so \(p\delta_D^{-1}\gamma \delta_D = p = p\,(\gamma)f\).
\item If \(p\in C\), then again \(p\delta_D^{-1}
 \in C\) so \(p\delta_D^{-1}\gamma \delta_D = p = p\varepsilon^{-1} \gamma \varepsilon\) since \(\varepsilon\) is the identity on
\(C\), so by the definition of \(f\), \(p\,(\gamma)f = p\delta_D^{-1}\gamma\delta_D\).
\item If \(p \in B\varepsilon \sqcup D\), then as \(B\varepsilon \sqcup D = (B \sqcup D)\varepsilon = (B \sqcup D)\delta_D\) we have \(p\delta_D^{-1} \gamma \delta_D = p\varepsilon^{-1}\gamma \varepsilon = p\,(\gamma)f\) as \(\delta_D\) and \(\varepsilon\) agree on \(B \sqcup D\),

\end{enumerate}
Similarly, for all \(\gamma\in \pstab_G(D)\) we have \((\gamma)f = \delta_C^{-1} \gamma\delta_C\). 

We have now shown 
\begin{eqnarray*}
(G) f &=&
\An{\pstab_G(C)f\cup \pstab_G(D)f} \\
&=&
\An{\pstab_G(C)^{\delta_D}\cup\pstab_G(D)^{\delta_C}}\\
&=&
\An{\pstab_G(C \cup A)\cup\pstab_G(D \cup A)}\\
&=&
\pstab_G(A)\\
&=& H
\end{eqnarray*}
\end{proof}

\subsection{Variants of full}\label{subsec:full}

In this subsection, we prove Theorem \ref{thm:aes} and Lemma \ref{lem:flexIsVigIfAF}.  These results explore interactions between variants of the concept of full, and what more can be deduced from being vigorous in the presence of variants of being full.

Recall the definitions of full and approximately full from Definitions~\ref{def:afull} and \ref{def:full}. Here, we offer yet another variant.
\begin{defn}
We will say a subgroup $G$ of $\autc$ is \emph{strongly approximately full} if and only if, for any $D,R\in \Cc$, \(\gamma_1, \ldots, \gamma_n\in G\) and clopen sets $D_1, \ldots, D_n$ partitioning $D$ such that the clopen sets $D_1\gamma_1, \ldots, D_n\gamma_n$ partition $R$, there exists $\chi$ in $G$ such that $\chi$ extends $\gamma_i|_{D_i}$ for all $1\leq i\leq n$.
\end{defn}

\begin{rmk}
    Note that strongly approximately full subgroups of $\autc$ are also approximately full. To appease the reader's certain curiosity we show in the next paragraph that the full group of the group generated by the full shift is not strongly approximately full. In particular not all full groups are strongly approximately full.
    
Consider elements of $\{0,1\}^\Z$ as bi-infinite words. Let $f:\{0,1\}^\Z \to \{0,1\}^\Z$ be the partial map which applies the full shift to those words that have a $0$ in the $0^{th}$ coordinate (to obtain those words that have a $0$ in the $1^{st}$ coordinate) and which pointwise fixes the set $S_{11}$ of those points with a $1$ in both the $0^{th}$ and $1^{st}$ coordinates.

Any homeomorphism of \(\{0, 1\}^{\mathbb{Z}}\) extending \(f\) must map the remaining domain set (those words with $1$ in the $0^{th}$ coordinate and $0$ in the $1^{st}$ coordinate) to the remaining codomain set (those words with $0$ in the $0^{th}$ coordinate and $1$ in the $1^{st}$ coordinate). 
In particular if such an extension is in the full group of the shift, then the element of $\{0,1\}^\Z$ which is $0$ on precisely the positive coordinates (and $1$ in all other coordinates) must be mapped to a word with $0$ in the $0^{th}$ position and $1$ in the $1^{st}$ position by a power of the full shift, which is impossible.
\end{rmk}

We can now give the proof of Theorem \ref{thm:aes}.  
In the proof of the second part below we in fact show that \vig{} subgroups of $\autc$ generated by their elements of small support are strongly approximately full, which is stronger than the stated result (see the corollary).  

\aes%%%%%%%%%%%%%%%%%%% This comment is to highlight the inclusion of this restated theorem.
\begin{proof}

First assume that $G$ is approximately full.

Let $\gamma$ be a non-identity element of $G$. We will find $\alpha, \beta \in G$ of small support such that $\alpha\beta = \gamma$.

Let $A\in \Cc$ be such that $A\gamma^{-1} \cap A = \varnothing$ and $(A\gamma^{-1} \sqcup A) \neq \C$. Let $(\gamma_1, \gamma_2, \gamma_3) = (\gamma, \gamma^{-1}, 1_G)$ where $1_G$ represents the identity of $G$. Set $D_1 := A\gamma^{-1}$ and $D_{2} := A$ and $D_{3} := \C \Sm (A\gamma^{-1} \sqcup A)$. Since $G$ is approximately full we may find $\beta$ in $G$ such that $\beta$ agrees with $\gamma$ on $D_{1}=A\gamma^{-1}$ and agrees with $1_G$ on $D_{3}=\C \Sm (A\gamma^{-1} \sqcup A)$. Set $\alpha:= \gamma\beta^{-1}$.% Note that $\alpha$ pointwise stabilises $A$.

Note firstly that $\gamma$ was arbitrary (non-identity) in $G$.  Also, $\alpha$ pointwise stabilises the clopen set $A\gamma^{-1}$, while $\beta$ pointwise stabilises the clopen set $D_3$.  Finally, as $\gamma=\alpha\beta$, we have point (\ref{aes:1}).

Now assume that $G$ is vigorous and is generated by its elements of small support. We will show that $G$ is strongly approximately full, from which it follows that $G$ is approximately full.

Let $D$ and $R$ be in $\Cc$. Let $\gamma_1, \ldots, \gamma_n\in G$ and let $D_1,\ldots ,D_n$ be a partition of $D$ into clopen sets such that \(D_1 \gamma_1, \ldots, D_n\gamma_n\) forms a partition of $R$.  Set the notation, for each $1\leq i\leq n$, $R_i\seteq D_i\gamma_i$.

Let $S,T\in\Cc$ be disjoint and such that  $S\subsetneq \C \Sm D$, $T\subsetneq \C\Sm R$, $(R\cup S\cup T)\subsetneq \C$, and $D\cup S\cup T\subsetneq \C$ (it is easy to check that such a pair of sets, $S$ and $T$ with the desired properties, exists). As $G$ is vigorous there is $\tau$ in $G$ with support contained in $R\cup S\cup T$ so that $R\tau \Se S$.

We will find $\chi \in G$ such that $\chi$ extends $(\gamma_i\tau)|_{D_i}$ for each $1\leq i \leq n$. This is sufficient because then $\chi\tau^{-1}$ will extend $\gamma_i|_{D_i}$ for each $1\leq i\leq n$ .

By construction, $D\sqcup S\subsetneq \C$. Choose  $E\in\Cc$, a proper clopen subset of $\C\Sm(D\sqcup S)$, noting  that $(D\sqcup R\tau)\cap E=\varnothing$ and $D\sqcup S\sqcup E \subsetneq \C$.  

Let $i\in \{1, \ldots, n\}$.  Set $I_i\seteq D_i\sqcup  R_i\tau$.  Set $K_i = I_i\sqcup E$.

Since $G$ is generated by its elements of small support, by Lemma \ref{lem:minime}, there is a clopen set $J_i$ with $I_i\subsetneq J_i\subsetneq K_i$ and an isomorphism $f_i:G\to \pstab_G(\C\Sm J_i)$ so that we have $(\gamma_i\tau) f_i$ agrees with $\gamma_i\tau$ on the set $I_i\cap (I_{i}\tau^{-1}\gamma_i^{-1})$.  Observe that $D_i\subseteq I_i\cap (I_{i}\tau^{-1}\gamma_i^{-1})$. Further observe that $(\gamma_i \tau)f_i$ acts as the identity on the complement of $K_i$, since $J_i\subsetneq K_i$. In particular $(\gamma_i \tau)f_i$ fixes $(D\sqcup R\tau)\Sm (D_i\sqcup R_i\tau)$ pointwise.  

For each  $i\in \{1, \ldots, n\}$, set $\chi_i\seteq(\gamma_i\tau)f_i$ .

Let $\chi\seteq \chi_1\chi_2\ldots \chi_n$.
 The element $\chi$ obtained has the property that, for each $i$, $\chi|_{D_i}=(\gamma_i\tau)|_{D_i}$, as desired.
\end{proof}

\begin{cor}\label{cor:vigappfulstrongappfull}
Let $G$ be a \vig{} approximately full subgroup of $\autc$. Then $G$ is strongly approximately full. In particular for \vig{} subgroups of $\autc$ approximately full and strongly approximately full are equivalent.
\end{cor}
\begin{proof}
By the first point of Theorem \ref{thm:aes} we have that $G$ is generated by its elements of small support. When proving the second part of Theorem \ref{thm:aes} we in fact showed that if $G$ is vigorous and generated by its elements of small support, then $G$ is strongly approximately full. It follows that our group $G$ is strongly approximately full.
\end{proof}

In the proof that (\ref{fessim}) $\Rightarrow$ (\ref{fesmat}) within the proof of Theorem \ref{thm:fes} that we shall give in the next subsection, we use the concept of a flexible group.  A group $G\leq \autc$ is \emph{flexible} if given  $U$, $V\in \Cc$, there is a group element $\gamma\in G$ so that $U\gamma\subseteq V$.  We observe in Lemma \ref{lem:flexIsVigIfAF} below that being flexible is a weaker property for a group to have than being vigorous, but we also show that in the presence of approximately full, the concepts are equivalent.

\begin{lem}\label{lem:flexIsVigIfAF}
Let $G\leq \autc$ be such that $G$ is approximately full. Then $G$ is flexible if and only if $G$ is vigorous.
\end{lem}
\begin{proof}
Let $G\leq \autc$.

If $G$ is vigorous it follows that $G$ is flexible. This follows because being vigorous allows one to move clopen sets into each other while restricting the set of other points being moved, while being flexible simply allows one to move clopen sets into each other.

Let us now suppose that $G$ is approximately full and flexible.  Let $A$, $B$, and $C\in \Cc$ with $B$ and $C$ proper subsets of $A$.  We will show there is some $\gamma\in G$ so that $B\gamma \subseteq C$ with the support of $\gamma$ fully contained in $A$. 

Choose $D\subseteq C$ with $B\cup D$ a proper clopen subset of $A$.

As $G$ is flexible, there is $\rho\in G$ so that $B\rho \subseteq A\backslash (B\cup D)$.  Consider the partition $X_1 = \{B, B\rho, \C\backslash (B\rho\cup B)\}$ of Cantor space into proper clopen sets.  Let us nominate these sets as $U_1\seteq B, U_2\seteq B\rho$, and $U_3\seteq \C\backslash (B\rho\cup B)$.  Set $\theta_1 \seteq \rho$, $\theta_2 \seteq \rho^{-1}$, and $\theta_3\seteq 1_G$ (where $1_G$ is the identity element in $G$).  The set $X_1$ corresponds to both the domain and range partitions appearing in the definition of approximately full, while the elements $\theta_1$, $\theta_2$, and $\theta_3$ are corresponding the group elements which respectively take $U_1$ to $U_2$, then $U_2$ to $U_1$, and finally $U_3$ to $U_3$.  Since $G$ is approximately full, there is $\chi_1\in G$ which agrees with $\theta_1$ on $U_1$ and $\theta_3$ on $U_3$.  We note that $B\chi_1\subseteq A\backslash (B\cup D)$ and $\chi_1$ acts as the identity over $\C\backslash (B\rho \cup B)$, and specifically, on the complement of $A$.

By an entirely similar argument, we can find a $\chi_2\in G$ which takes $A\backslash (B\cup D)$ into $D$ while pointwise fixing the complement of $A$.  The composition $\gamma=\chi_1\chi_2$ acts as the identity on the complement of $A$ and $B\gamma \subseteq D\subseteq C$.  In particular, $G$ is vigorous.
\end{proof}
\subsection{Equivalent conditions}\label{subsec:equiv}
In this subsection, we give the proof of Theorem \ref{thm:fes}.

We begin with more lemmata concerning vigorous subgroups of $\autc$.

\begin{lem} \label{lem:minime2}
Let $G$ be a \vig{} subgroup of $\autc$ generated by its own set of elements of small support. Let $I$ and $K$ in $\Cc$ be such that $I \subsetneq K$. Then there exists $J \in \Cc$ and an isomorphism $f:G \tooo \pstab_G(\C \Sm J)$ such that $I \subsetneq J \subsetneq K$ and for each $\gamma \in G$ the homeomorphism $(\gamma) f$ agrees with $\gamma$ on $I\gamma^{-1} \cap I$. 
\end{lem}
\begin{proof}
Our strategy is to first find sets $B$, $C$, and $D$ and a group element $\delta$ as in Lemma \ref{lem:minime}.  We then determine the clopen set $J$ with $I\subsetneq J \subsetneq K$ so that the $\delta$-induced isomorphism $f$ has $f:G \tooo \pstab_G(\C \Sm J)$.    Finally, for each $\gamma\in G$ and $p\in I\cap I{\gamma^{-1}}$ we verify that $p\gamma f = p \gamma$.

Let $L,L' \in \Cc$ be such that $L' \subsetneq L \subsetneq K \Sm I$.

We set $D\seteq I$ and $B \seteq L \sqcup (\C \Sm K)$.  Note $B\cap D = \emptyset$.  This allows us to determine $C$ as the complement of $B$ and $D$ in $\C$.  That is, $C\seteq \C\Sm (B\sqcup D)$.  Observe as well that as $L'\subsetneq L \subsetneq B$ we have $L$ and $L'$ are each disjoint from $C\sqcup D$.

Since $G$ is \vig{} we may find $\delta \in \pstab_G(I)$ so that $B\delta \Se L'$.

Set $A\seteq B \Sm B\delta$ and $J\seteq \C \Sm A$.  Note that 
\[
J=B\delta\sqcup C\sqcup D\subseteq L'\sqcup C\sqcup D\subsetneq L\sqcup C\sqcup D=K.
\]That is, $J\subsetneq K$.  Further note that as $D=I$ and $C$ is not empty we have $I\subsetneq J$, and we conclude that $I\subsetneq J \subsetneq K$.

We are now in the context of Lemma \ref{lem:minime}.  Let $f$ be the induced homomorphism from $\autc$ to $\autc$ which restricts to an isomorphism from $G$ to $\pstab_G(A)$.  

Recall there is a monic function $\varepsilon:\C\to\C$ which acts as the identity on $C$ and $D$, and acts as $\delta$ on $B$, and where $\gamma f = \gamma^{\varepsilon}$ on $J$ (Note that $\gamma f$ acts as the identity off \(J\)).

Now, for all points $p\in I\cap I{\gamma^{-1}}$,  we can compute $p \gamma f$.  Specifically, $p\gamma f= p\varepsilon^{-1}\gamma\varepsilon= p\gamma\varepsilon$ since $\varepsilon$ acts as the identity in $I\cap I\gamma^{-1}\subseteq D$.  As $p\in I\gamma^{-1}$ we see that $p\gamma=z\in I = D$ so we can further compute  $p\gamma f=p\gamma\varepsilon=z\varepsilon=z$.  Thus, $\gamma f$ and $\gamma$ agree over $I\cap I\gamma^{-1}$.

  %Set $J \seteq \C \Sm (L\delta \Sm L)$. Let $\varepsilon$ be the map from $\C$ which agrees with $\delta$ on $\C \Sm L$ and pointwise stabilises $L$. Let $f$ be the map from $G$ which maps $\gamma \in G$ to the  element of $\autc$ which agrees with $\varepsilon^{-1}\gamma\varepsilon$ on $\C\varepsilon$ and pointwise stabilises $\C \Sm \C\varepsilon$.

\end{proof}
\begin{lem}\label{lem:bignormal}
    Suppose that \(G\) is vigorous, \(N\) is a non-trivial normal subgroup of \(G\), \(A\in K_{\mathfrak{C}}\) and \(\delta\in G\) is supported on \(A\).
    There is some \(\gamma\in N\) such that \(\gamma|_A=\delta|_A\).
\end{lem} 
\begin{proof}

    % We first show that \(N\) contains elements of small support. Let \(g\in N\backslash\{1_N\}\) and let \(F\in \Cc\) be such that \((F)g\cap F= \varnothing\) and \((F)g \cup F \neq \mathfrak{C}\). Let \(h\in G\) be any non-trivial element supported on \(F\) (which exists as \(G\) is vigorous). Then
    % \[h^{-1}g^{-1}hg\in N\]
    % is supported on \(F\cup Fg\). In particular \(N\) contains elements of small support.

    Let \(k\in N\backslash \{1_N\}\). Let \(B\in K_{\mathfrak{C}}\) be such that \((B)k\cap B = \varnothing\).

    Let \(l\in G\) be such that \((A)l\subseteq H\), so \(A\subseteq (H)l^{-1}\). It follows that \(k^{l^{-1}}\in N\), and \((A)k^{l^{-1}}\cap A = \varnothing\). Finally, we have 
    \[\gamma:=[k^{l^{-1}}, \delta]\in N\]
     agrees with \(\delta\) on \(A\) as required.
\end{proof}
\begin{lem}\label{lem:normalvig}
    If \(G\) is vigorous and \(N\) is a non-trivial normal subgroup of \(G\), then \(N\) is vigorous.
\end{lem} 
\begin{proof}
    Suppose that $A,B,C$ are clopen subsets of $\C$ with $B$ and $C$ proper (non-empty) subsets of $A$. We need only show that there exists $\gamma$ in $N$ with $\suppp{\gamma} \Se A$ and $B\gamma \Se C$.
    Let \(D, E\in \Cc\) be such that \(A=B\sqcup D\sqcup E\). Without loss of generality assume that \(C\cap (B\sqcup D) \neq \varnothing\). Define \(A':=B\sqcup D\) and \(C':=C\cap A'\).

    As \(G\) is vigorous, there is some \(\delta\in G\) supported on \(A'\) with \((B)\delta \subseteq C'\) (if \(C'=A'\) then the identity function does this). 

  By Lemma~\ref{lem:bignormal}, there is some \(\gamma \in N\) such that \(\gamma|_{\mathfrak{C}\backslash E}=\delta|_{\mathfrak{C}\backslash E}\). So \(B\gamma=B\delta\subseteq C'\subseteq C\), and \(\gamma\) is supported on \(A\) as required.
    
\end{proof}
The following lemma is proven following the outline of the proof of the first part of Theorem \ref{thm:aes}.

Recall that a subgroup $G$ of $\autc$ is approximately full if and only if whenever $\gamma_1, \ldots, \gamma_n\in G$ and $D_1, \ldots, D_n$ are clopen sets partitioning $\C$ such that $D_1 \gamma_1, \ldots, D_n \gamma_n$ also partitions $\C$ and $j\in \{1, \ldots, n\}$ then there exists $\chi$ in $G$ such that $\chi$ extends
$\gamma_i|_{D_i}$ for each $i \in \{1, \ldots, n\} \Sm \{j\}$. 
\begin{lem}\label{lem:normBySmallSupport}
Let \(G\) be vigorous and approximately full. If \(N\) is a non-trivial normal subgroup of \(G\), then \(N\) is vigorous and approximately full. Moreover every element of \(N\) can be expressed as a product of two elements of \(N\) with small support.
\end{lem}
\begin{proof}
% Let $\gamma_1, \ldots, \gamma_n\in N$ and $D_1, \ldots, D_n$ clopen sets partitioning $\C$ such that $D_1\gamma_1, \ldots, D_n\gamma_n$ also partition $\C$ and $j\in \{1, \ldots, n\}$.

% As \(N\subseteq G\) then there exists $\chi$ in $G$ such that $\chi$ extends
% $\gamma_i|_{D_i}$ for each $i\neq j $.

By Lemma~\ref{lem:normalvig}, \(N\) is vigorous. By Theorem \ref{thm:aes}, \(G\) generated by its elements of small support, and we need only show that every element of \(N\) can be expressed as a product of two elements of \(N\) with small support.
%Suppose \(G\) is full and vigorous and \(N\lhd G\) is normal in \(G\).

Let $\gamma$ be a non-identity element of $N$. We will find \(\alpha, \beta\in N\) with small support such that $\alpha\beta = \gamma$.

Let $A\in \Cc$ be such that $A\gamma^{-1} \cap A = \varnothing$ and $(A\gamma^{-1} \sqcup A) \neq \C$. Let $(\gamma_1, \gamma_2, \gamma_3) = (\gamma, \gamma^{-1}, 1_G)$ where $1_G$ represents the identity of $G$. Set $D_1 := A\gamma^{-1}$ and $D_{2} := A$ and $D_{3} := \C \Sm (A\gamma^{-1} \sqcup A)$. Since $G$ is approximately full we may find $\delta$ in $G$ such that $\delta$ agrees with $\gamma$ on $D_{1}=A\gamma^{-1}$ and agrees with $1_G$ on $D_{3}=\C \Sm (A\gamma^{-1} \sqcup A)$. 

Note that \(\delta\in \pstab_G(D_3)\). Let \(B\) be a proper clopen subset of \(D_3\). By Lemma~\ref{lem:bignormal}, there is \(\beta\in N\) such that \(\beta|_{B\cup D_1 \cup D_2}=\delta|_{B\cup D_1 \cup D_2}\).
Set $\alpha:= \gamma\beta^{-1}$.% Note that $\alpha$ pointwise stabilises $A$.

Note firstly that $\gamma$ was an arbitrary (non-identity) element of $N$.  Also, $\alpha$ pointwise stabilises the clopen set $A\gamma^{-1}$, while $\beta$ pointwise stabilises the clopen set $B $.  So \(\gamma\) is a product of elements of small support in \(N\) as required.

\end{proof}
Below, working towards Lemma~\ref{lem:MatuiSimplicity}, we essentially follow Matui's proof of his Theorem 4.16 of \cite{Matui15}, but in our context.  The core result is that the commutator subgroup of an approximately full vigorous group is simple, but we write it slightly differently to capture a chain of natural constructions.

\begin{lem}\label{lem:MatuiSimplicity_2}
Suppose that G is vigorous and approximately full. If \(N \unlhd  D(G)\), then \(N \unlhd  G\).
\end{lem}
\begin{proof}
This proof is similar to that of Theorem 4.15 from \cite{Matui15} but is included for completeness and due to the need to translate the proof to our context.

Let \(\alpha\in G\) and \(\tau\in N\). We show that \(\tau^{\alpha} \in N\).
We may assume that \(\alpha\) and \(\tau\) have small support as by Lemma~\ref{lem:normBySmallSupport}, both \(G\) and \(N\) are generated by their elements of small support. 
    As \(G\) is vigorous, there exists \(\sigma \in G\) such that \((\supt{\alpha})\sigma \cap \supt{\tau} = \varnothing\). 
    Then 
    \[\tau^\alpha=(\tau^{(\alpha^{-1})^{\sigma}})^{\alpha} = \tau^{[\sigma, \alpha]}  \in N.\]
\end{proof}

\begin{lem}\label{lem:MatuiSimplicity}
The commutator subgroup $D(G)$ of a vigorous, approximately full group \(G\) is vigorous, approximately full, and simple.
\end{lem}
\begin{proof}
The non-simplicity parts are handled by Lemma~\ref{lem:normBySmallSupport}
The simplicity proof is taken directly from \cite{Matui15} Theorem 4.16 but is included for completeness and due to the need to slightly translate the proof to out context.

Let \(N\) be a non-trivial normal subgroup of \(D(G)\). Let \(\tau \in N \backslash \{1\}\).
There exists a non-empty clopen set \(A \subseteq \mathfrak{C}\) such that \(A \cap  (A)\tau = \varnothing\).
We would like to show that \([\alpha, \beta]\) is in \(N\) for any \(\alpha, \beta \in G\).

First, we assume
\(\alpha,\beta\) have small support. 
As \(G\) is vigorous, we can find \(\gamma \in G\) such that
\((\supt{\alpha})\gamma \cap \supt{\beta} = \varnothing\) and \(\supt{\alpha}\cup \supt{\gamma}\) is not dense. 
There exists \(\sigma \in G\) such that
\((\supt{\alpha} \cup \supt{\gamma})\sigma \subseteq A\).
By Lemma~\ref{lem:MatuiSimplicity_2}, \(\tilde{\tau} = \tau^\sigma \in N\).

By the construction of \(\tilde{\tau}\), 
\((\supt{\alpha} \cup \supt{\gamma}) \cap (\supt{\alpha} \cup \supt{\gamma})\tilde{\tau} = \varnothing\). Hence \(\tilde{\gamma} = (\gamma^{-1})^{\tilde{\tau}}\gamma=[\tilde{\tau},\gamma ]\) agrees with \(\gamma\) on \((\supt{\alpha} \cup \supt{\gamma})\).
In particular it satisfies \((\supt{\alpha})\tilde{\gamma} \cap
\supt{\beta} = \varnothing\). It follows that \( \alpha^{ \tilde{\gamma}}\)
commutes with \(\beta\) and \(\beta^{-1}\). Moreover, by Lemma~\ref{lem:MatuiSimplicity_2}, \(\tilde{\gamma}\) is in
\(N\), too. 
Therefore
\[[\alpha, \beta] = \alpha^{-1}\beta^{-1} \alpha\beta = \alpha^{-1}\alpha^{\tilde{\gamma}}\beta^{-1} (\alpha^{\tilde{\gamma}})^{-1}\alpha\beta =[\alpha, \tilde{\gamma}]\beta^{-1}[\alpha, \tilde{\gamma}]^{-1}\beta\in N.\]

Next, we show \([\alpha, \beta] \in N\) when only \(\beta\) is assumed to have small support. 
By Lemma~\ref{lem:normBySmallSupport}, we can find \(\alpha_1, \alpha_2, \beta_1, \beta_2 \in G\) with small support
such that \(\alpha = \alpha_1 \alpha_2\) and \(\beta=\beta_1\beta_2\). By the proof above, \([\alpha_1, \beta]\)
and \([\alpha_2, \beta]\) are in \(N\). Note that

\[[\alpha, \beta] = [\alpha_1\alpha_2, \beta]=\alpha_2^{-1}\alpha_1^{-1}\beta^{-1}\alpha_1(\beta\alpha_2\alpha_2^{-1}\beta^{-1})\alpha_2\beta= [\alpha_1, \beta]^{\alpha_2} [\alpha_2, \beta],\]
and similarly
\[[\alpha_1, \beta] =[\alpha_1, \beta_2][\alpha_1, \beta_1]^{\beta_2}, \quad \quad [\alpha_2, \beta] =[\alpha_2, \beta_2][\alpha_2, \beta_1]^{\beta_2}.\]
So from Lemma~\ref{lem:MatuiSimplicity_2}, we have \([\alpha, \beta]\in N\) as required.
\end{proof}

\fes%%%%%%%%%%%%%%%%%%%%%%%%%%%%%%%%%%%%%%%%%%%%%%%%%%%%

\begin{proof}
\hspace{0pt}

(\ref{fessim}) $\Rightarrow$ (\ref{fesfls})

Let \(w\) be a non-trivial freely reduced word. By Lemma \ref{FiN} \(w(G)_\circ\) is normal in \(G\). Theorem \ref{thm:vigfree} gives that \(w(G)_\circ\) contains a non-identity element of \(G\), so \(w(G)_\circ\) is a non-trivial normal subgroup of \(G\) which must therefore be \(G\) by simplicity.

(\ref{fesfls}) $\Rightarrow$ (\ref{feslfl})

By inspection, the property of being $\fls{}$ is stronger than Condition (\ref{feslfl}).

(\ref{feslfl}) $\Rightarrow$ (\ref{fespss})

The group $G$ is perfect because it is generated by a subset of the set of commutators. Noting that the complement of a proper clopen set is a proper clopen set, we see that $\Ss{[\alpha,\beta]}{\alpha,\beta\in G,\exists K\in\Cc\,\text{such that}\, \supt{[\alpha,\beta]}\subseteq K}$ is equivalent to the set $\Ss{[\alpha,\beta] \in \pstab_G(A)}{\alpha,\beta \in G, A \in \Cc}$.  Therefore this last set generates $G$ and is a subset of the set of elements of small support.

(\ref{fespss}) $\Leftrightarrow$ (\ref{fespaf})

This follows immediately from Corollary \ref{cor:aes}.

(\ref{fespaf}) $\Rightarrow$ (\ref{fessim})

 This follows from Lemma~\ref{lem:MatuiSimplicity}.
% Let $\delta$ be a non-identity element of $G$. We will show that the normal closure of $\delta$ is equal to $G$. 

% As $G$ is generated by its elements of small support it is sufficient to show that the set $\Ss{\gamma \in \pstab_G(C)}{C\in \Cc}$ is a subset of the normal closure of $\delta$.

% Let $I$ be in $\Cc$.  Observe that $\C\backslash I \in \Cc$ as well.  Specifically, as $I$ is arbitrary it is therefore sufficient to show that $\pstab_G(\C \Sm I)$ is a subset of the normal closure of $\delta$.

% By Lemma \ref{lem:minime2} we may find $J \in \Cc$ such that $J \supseteq I$ and $\pstab_G(\C \Sm J)$ is isomorphic to $G$. It follows then that the group $\pstab_G(\C \Sm J)$ is perfect.

% Let $\mu$ and $\nu$ be in $\pstab_G(\C \Sm J)$. Since $\pstab_G(\C \Sm J)$ is perfect it is sufficient to show the commutator $[\mu,\nu]$ is in the normal closure of $\delta$.

% Let $P$ in $\Cc$ be such that $P\delta \cap P$ is empty (note that $P$ exists by Lemma \ref{ntms}) and further let $\lambda \in G$ be such that $J\lambda \subseteq P$. Since $\suppp{[\delta,\mu^\lambda]} \cap \suppp{\nu^\lambda} \Se P$ and both $[\delta,\mu^\lambda]$ and $\nu^\lambda$ setwise stabilise $P$ and $[\delta,\mu^\lambda]$ agrees with $\mu^\lambda$ on $P$ it follows that $[[\delta,\mu^\lambda],\nu^\lambda]^{(\lambda^{-1})} = [\mu^\lambda,\nu^\lambda]^{(\lambda^{-1})}=[\mu,\nu]$.

% However, $[[\delta,\mu^\lambda],\nu^\lambda]^{(\lambda^{-1})}$ is in the normal closure of $\delta$.  Thus, we have that $[\mu,\nu]$ is in the normal closure of $\delta$ as desired.

(\ref{fesmat}) $\Rightarrow$ (\ref{fessim}) 

This follows from Lemma~\ref{lem:MatuiSimplicity}.  

(\ref{fessim}) $\Rightarrow$ (\ref{fesmat}) 

This direction was established by Belk and Matucci in an independent project.  Upon hearing of this work, they graciously suggested its inclusion here.

% The argument of Matui in Theorem 4.16 of  \cite{Matui15} shows that the commutator subgroup of a vigorous full group of homeomorphisms of Cantor space is simple.  Note that Matui does not use our language of  `vigorous' in \cite{Matui15}, he instead assumes that a groupoid is purely infinite and minimal. However Matui's argument still goes through as by Proposition 4.11 of his paper the transformation groupoid of a vigorous group will always have these properties. In passing we note that Matui also uses essentially principal Et\'ale groupoids instead of group actions, and in the case of transformation groupoids his context allows only for countable groups with ``almost" free actions (although his arguments work in our context, regardless of these facts).

Let us now suppose that $G$ is simple and $H$ is the full group of $G$ in $\autc$.  As $G$ is vigorous, $G$ satisfies all of the conditions (\ref{fessim})--(\ref{fespaf}).  We now argue that $G$ is the commutator subgroup of its own full group within $\autc$.

Since $G$ is perfect and a subgroup of the full group of $G$ it follows that $G$ must be a subgroup of the commutator subgroup of the full group of $G$. Recalling as before that Matui in \cite{Matui15} shows that the commutator subgroup of a full vigorous (flexible) group of homeomorphisms of Cantor space is simple, it is sufficient to show that $G$ is normal in the commutator subgroup of the full group of $G$. We will now show that $G$ is normal in the full group of $G$.

Since $G$ is generated by its elements of small support we only need to show that the elements of small support in $G$ are closed under conjugation by elements of the full group of $G$.  

Let $\gamma$ now represent an element of small support in $G$, which is supported inside a particular clopen set $U\in \Cc$.  Let $\sigma$ be a (non-identity) element of the full group $H$.  Then, by definition, there is a positive integer $k$ and partitions $D\seteq \{D_1,D_2,\ldots, D_k\}$ and $R\seteq\{R_1,R_2,\ldots ,R_k\}$ of $\C$, together with a set of elements $\{\sigma_1,\sigma_2,\ldots,\sigma_k\}\subsetneq G$ so that for each index $i$ we have $\sigma$ agrees with $\sigma_i$ over $D_i$ and $D_i\sigma = R_i$.  Further, there is an index $j$ so that $D_j\not\subseteq \supt{\gamma}$.  We can now choose $E_j,F_j\in \Cc$ so that $E_j\cap F_j=\varnothing$, $E_j\cup F_j = D_j$, and $F_j\cap \supt{\gamma}=\varnothing$.

Now the sets $$D'\seteq\{D_1,D_2,\ldots, D_{j-1},E_j, F_j, D_{j+1}, D_{j+2},\ldots,D_k\}$$ and $$R'\seteq\{R_1,R_2,\ldots,R_{j-1}, E_j\sigma_j, F_j\sigma_j, R_{j+1}, R_{j+2},\ldots,R_k\}$$ are partitions of $\C$ into clopen sets, and further we have that $\sigma$ agrees with $\sigma_i$ over each $D_i$, for any index $i$, and particularly $\sigma$ agrees with $\sigma_i$ over each of the sets $E_j$ and $F_j$.  Thus we have a system as per the definition of approximately full groups, as each $\sigma_i\in G$.

Now, we note that the union of these restricted elements is $\sigma$, but we can find an element $\chi\in G$ which agrees with all of these restricted elements over all of the parts of the domain partition $D'$ except $F_j$.  It follows that (as $\gamma$ acts as the identity over $F_j$) $\gamma^\sigma = \gamma^{\chi}$, but this last element is an element of $G$.

As $\gamma$ was an arbitrary element of small support in $G$, and as $\sigma$ was an arbitrary (non-identity) element of $H$, we see that the set of elements of small support in $G$ is not only a generating set for $G$, but this set is also a union of conjugacy classes of $H$.  Consequently, $G$ is normal in $H$, as desired.
\end{proof}

We note that there exist full subgroups of $\autc$ which are not simple. For example the Higman--Thompson group $V_3$.

\section{Exploring simple vigorous groups}\label{sec:exploreSimpleVig}

Recall $\K$ is the family of simple \vig{} subgroups of $\autc$.  In the first subsection we show this is a rather large family.  

Further recall that $\Kf$ is the subfamily of $\K$ consisting of the finitely generated simple \vig{} subgroups of $\autc$.
From Lemma \ref{Vcsg} and the fact that the commutator subgroups of the Higman--Thompson groups $V_k$ are vigorous, finitely presented, and simple, we obtain for instance that the (now called) simple R\"over-Nekrashevych groups of \cite{Nekrashevych} all in $\Kf$.  It is also easy to verify that the Brin--Thompson groups $nV$ are vigorous, so as these are finitely presented \cite{Brin,BrinPres,HennigMatucci} they are in $\Kf$ as well.  In the second subsection we prove Theorem \ref{sub2} that all of these groups are two generated by torsion elements.
\subsection{Exploring the classes $\K$ and $\Kf$}\label{subsec:exploringK}

Having given some examples previously of groups in $\Kf$ (e.g., R. Thompson's group $V_2$ and R\"over's group $V(\Gamma$)),  we will now show that $\K$ and $\Kf$ are closed under various natural constructions to allow for more groups in $\K$ and $\Kf$ to be easily constructed.

{For $D$ a non-empty subset of $\C$ we will say that a subgroup $G\leq \autc$ is \emph{vigorous over $D$} if for any non-empty clopen subset $A$ of $D$ and non-empty proper clopen subsets $B$ and $C$ of $A$ there exists $\gamma \in G$ with support contained in $A$ such that $B\gamma \subseteq C$.  We will write $\Kk{D}$ for the family of simple subgroups of $\autc$ whose elements are with support contained in $D$ and where these groups are vigorous over $D$.  We will write $\Kkf{D}$ for the family of finitely generated groups in $\Kk{D}$. }

\begin{rmk} \label{rmk:KcC}
Let $C$ and $D$ be non-empty clopen subsets of $\C$ and let $\lambda:C \tooo D$ be a homeomorphism. The homeomorphism $\lambda$ induces an isomorphism from $\pstab_{\autc}(\C \Sm C)$ to $\pstab_{\autc}(\C \Sm D)$. This isomorphism induces a correspondence between groups in $\Kk{C}$ and groups in $\Kk{D}$ and a correspondence between groups in $\Kkf{C}$ and groups in $\Kkf{D}$. Note that $C$ and $D$ could be equal to $\C$. 
\end{rmk}

The above remark is both an easy observation and a useful tool for constructing groups trivially in $\K$ or $\Kf$ and, when combined with the propositions below, groups non-trivially in $\K$ or $\Kf$.

\begin{prop} \label{prop:KcU}
Let $U$ and $V$ in $\Cc$ be such that $U \cap V$ is non-empty. If $G$ is in $\Kk{U}$ and $H$ is in $\Kk{V}$ then $\An{G \cup H}$ is in $\Kk{U \cup V}$. If furthermore $G$ is in $\Kkf{U}$ and $H$ is in $\Kkf{V}$ then $\An{G \cup H}$ is in $\Kkf{U \cup V}$.
\end{prop}
\begin{proof}%%%%%%%%%%%%
Let $A$ be a non-empty clopen subset of $U \cup V$ and let $B$ and $C$ be non-empty proper subsets of $A$.

Let $D \subseteq A \Sm B$ not be a superset of $C$. Since $D \subsetneq A \subseteq U \cup V$ at least one of $D \cap U$ and $D \cap V$ are non-empty. By symmetry we may assume that $D \cap U$ is non-empty. Let $\lambda \in G$ be such that $(U \Sm V)\lambda \subseteq D$. Note that $V\lambda \supseteq (U \cup V) \Sm D$  

By Remark \ref{rmk:KcC} the group $H^\lambda$ is in $\Kk{V\lambda}$. Since $B$ and $C \Sm D$ are non-empty proper subsets of $A \Sm D \subseteq V\delta \cap A$ there exist $\gamma \in \pstab_{H^\lambda}(\C \Sm A)$ such that $B\gamma \subseteq C \Sm D$ as desired.

For the remainder of the argument, note that Theorem \ref{thm:fes} holds for groups wholly supported on an element $D\in \Cc$ (by replacing `vigorous' with `vigorous over $D$'), as such sets $D$ are themselves homeomorphic to Cantor space.  

Now we will show that $\An{G \cup H}$ is simple. By Theorem \ref{thm:fes} the groups $G$ and $H$ are \fls{}. Also by Theorem \ref{thm:fes} it is sufficient to show that $\An{G \cup H}$ is \fls{} as we have just proved that the group $\langle G\cup H\rangle$ is vigorous over $U\cup V$. Since for any non-trivial freely reduced word \(w\) we have $\langle G\cup H\rangle=\langle w[G]_\circ \cup w[H]_\circ\rangle$ we see that $\langle G\cup H\rangle$ is also \fls{} and we are done.

If $G$ and $H$ are finitely generated then it is immediate that $\An{G \cup H}$ is also finitely generated.
\end{proof}

The next lemma is used both in the next Proposition and in Theorem \ref{thm:sub2}.
Observe that it follows immediately from Lemma \ref{lem:ACD} and Remark \ref{rmk:KcC}.
\begin{lem} \label{lem:SoP}
Let $U$ and $V$ in $\Cc$ be such that $U \cap V$ is non-empty. If $G$ pointwise stabilises $\C\backslash (U \cup V)$, and is generated by its elements which are supported on a proper clopen subset of $U \cup V$, and $G$ is vigorous over $U\cup V$ then 
\[\An{\pstab_G(U \Sm V) \cup \pstab_G(V \Sm U)} = G.
\]
\end{lem}

\begin{prop} \label{prop:psk}
Let $K$ be in $\Cc$. If $G$ is in $\K$ then the group $\pstab_G(\C \Sm K)$ is in $\Kk{K}$. If $G$ is in $\Kf$ then the group $\pstab_G(\C \Sm K)$ is in $\Kkf{K}$.
\end{prop}
\begin{proof}
Let $I_1 \in \Cc$ and $I_2 \in \Cc$ be such that $I_1 \cap I_2$ is non-empty, $I_1\not \subseteq I_2$, $I_2\not\subseteq I_1$, and $I_1 \cup I_2 = K$.

By Lemma \ref{lem:minime2} there exists $J_1$ and $J_2$ in $\Cc$ such that $I_1 \subsetneq J_1 \subsetneq K$ and $I_2 \subsetneq J_2 \subsetneq K$ and both $\pstab_G(\C \Sm J_1)$ and $\pstab_G(\C \Sm J_2)$ are isomorphic to $G$. The groups $\pstab_G(\C \Sm J_1)$ and $\pstab_G(\C \Sm J_2)$ are therefore simple. Since $G$ is \vig{} the group $\pstab_G(\C \Sm J_1)$ is in $\Kk{J_1}$. Similarly the group $\pstab_G(\C \Sm J_2)$ is in $\Kk{J_2}$. By Proposition \ref{prop:KcU} the group $\An{\pstab_G(\C \Sm J_1) \cup \pstab_G(\C \Sm J_2)}$ is in $\Kk{K}$.

Similarly, if $G$ is finitely generated then as  $\An{\pstab_G(\C \Sm J_1) \cup \pstab_G(\C \Sm J_2)}$ is generated by two isomorphic copies of $G$ then the group $\An{\pstab_G(\C \Sm J_1) \cup \pstab_G(\C \Sm J_2)}$ is finitely generated and is in $\Kkf{K}$.  

It remains to show that $\An{\pstab_G(\C \Sm J_1) \cup \pstab_G(\C \Sm J_2)}= \pstab_G(\C\backslash K)$.

As $G$ is vigorous we see that $H=\pstab_G(\C\backslash K)$ is vigorous over $K$.  By Lemma \ref{lem:ssgp} we have that $H$ is generated by those of its elements which are supported on proper clopen subsets of $K$. We apply Lemma \ref{lem:SoP} using $J_1$ and $J_2$ as $U$ and $V$, respectively, and noting that
\[
\pstab_{G}(\C\backslash J_1)=\pstab_{H}(K\backslash J_1)=\pstab_{H}(J_2\backslash J_1)
\]
\centerline{and}
\[
\pstab_{G}(\C\backslash J_2)=\pstab_{H}(K\backslash J_2)=\pstab_{H}(J_1\backslash J_2)
\] hold to conclude that $H=\langle \pstab_{H}(J_1\backslash J_2)\cup \pstab_{H}(J_2\backslash J_1)\rangle$ as desired.
\end{proof}

\begin{prop} \label{prop:KcA}
Let $G$ be in $\K$. Let $D$ be in $\Cc$. Let $\gamma \in G$ be such that $D\gamma \cap D$ is empty and $D\gamma \cup D$ is a proper subset of $\C$. Let $\delta$ be in $\pstab_{\autc}(\C \Sm D)$. Let $H$ be the group $\An{G \cup \{[\delta,\gamma]\}}$. Then $H$ is also in $\K$. If further $G$ is in $\Kf$ then $H$ is also in $\Kf$.
\end{prop}
\begin{proof}
The group $H$ is $\vig{}$ because it contains \(G\). If the group $G$ is finitely generated then the group $H$ is also finitely generated. It remains to show that $H$ is simple.

Let $T$ be equal to the union
\[\bigcup\limits_{K \in \Cc} \pstab_H(K).\]

Note that as \([\delta,\gamma]\) fixes \(\C \backslash (D\cup D\gamma )\) we have \([\delta,\gamma]\in T\).

Since $G$ is generated by $T \cap G$ and $[\delta,\gamma]$ is in $T$ it follows that $H$ is generated by $T$. Therefore by Theorem \ref{thm:fes} it is sufficient to show that $H$ is perfect.

By Theorem \ref{thm:fes} the group $G$ is perfect so it is sufficient to show that $[\delta,\gamma]$ is a commutator of elements of $H$ (we cannot use $\delta$ and $\gamma$ because $\delta$ may not be in $H$).

Let $A$ and $B$ in $\Cc$ be disjoint subsets of $\C \Sm (D\gamma \cup D)$. By Theorem \ref{thm:aes} the group $H$ is approximately full so we may find $\mu$ and $\nu$ in $H$ such that $\mu$ agrees with $\delta$ on $D$ and is wholly supported on $D \cup A$ (for this, we can actually take $\mu=[\gamma,\delta]$, which agrees with $\delta$ on $D$) and so that $\nu$ agrees with $\gamma$ on $D$ and is wholly supported on $D\gamma \cup D \cup B$. (For this latter construction, we can use the approximately full property taking $\nu$ so that $\nu$ agrees with $\gamma^{-1}$ on $D\gamma$ to construct our element.) Now $[\mu,\nu] = [\delta,\gamma]$. 
\end{proof}

\begin{cor}
All the groups in $\Kf$ are semigroup $2$-generated. This includes $V_2$ and any group that may be constructed from $V_2$ using the constructions described in Remark \ref{rmk:KcC}, Proposition \ref{prop:KcU}, Proposition \ref{prop:psk} and Proposition \ref{prop:KcA}.
\end{cor}
\begin{proof}
This follows immediately from %Proposition \ref{KcC}, Proposition \ref{KcU}, \linebreak Proposition \ref{KcA} and 
Theorem \ref{thm:sub2}.
\end{proof}

Recall in Example \ref{VFP} we construct a non-simple overgroup in $\autc$ of the Higman--Thompson group $V_6$. Consequently neither $\K$ nor $\Kf$ are closed under taking overgroups in $\autc$.

\subsection{Two-generation of groups in $\Kf$} \label{c2p}
In this section we will prove
Theorem \ref{sub2}.

As has been mentioned above, the groups we are exploring bear some comparison with the full groups explored by Matui and others (see, e.g., \cite{Matui06,Matui12,Matui13,Matsumoto15,MatsumotoMatui17}). Matui's full groups are generally built from almost finite or purely infinite \'etale groupoids and act naturally on the groupoids' unit spaces (which is a Cantor space in many cases of interest). However, the groups we  consider need not be full, instead they are vigorous and  approximately full. Meanwhile, Matui's groups need not be vigorous, but often act with minimal actions.  In particular, neither of these classes contains the other.  In general, our class of groups corresponds most closely with the case where Matui is analyzing full groups over an \'etale groupoid of purely infinite type, and some similar statements can be made (see, e.g., Lemma 4.14 and Theorem 4.16 \cite{Matui13} which are close to our Lemmas \ref{lem:normBySmallSupport} and \ref{lem:MatuiSimplicity}, respectively).

Continuing this theme, in what follows we build an abelian group $(\Ooo{G},+)$ corresponding to Matui's $0$-th homology group of \cite{Matui12}, which in turn is based on Crainic and Moerdijk's homology theory for \'{e}tale groupoids (see \cite{CrainicMoerdijk}). Indeed, there have been many related constructions since \cite{Matui12}, for example, see \cite{Nawata12,Matui13,Matsumoto15,Matui15,Nekrashevych,Lawson16,MatsumotoMatui17,MatteBon17}.

In our case, we use our group $(\Ooo{G},+)$ to help us to understand how to assemble torsion elements which have some prescribed dynamics (see Lemmas \ref{lem:eag} and \ref{lem:part}) in approximately full groups.  

\subsubsection{Homology}\label{subsubsec:homology}

Our proof that a finitely generated vigorous simple group is actually two generated uses our constructed group $(\Ooo{G},+)$.  An older proof using only elementary arguments (but, this path is more compicated), can be found in \cite{HydeDiss}.

%\begin{exmp} \label{exm:Kncusg}
%Let $\alpha_{\{0,1\}}$ be as in Example \ref{Fex}. We will use $G$ for the group
%\[\An{\It{\{0,1\}}\alpha_{\{0,1\}} \cup V_2}.\]
%The group $G$ is a supergroup of $V_2$ which is in $\Kf$. We will show that $G$ is not simple and is therefore not in $\K$.

%We will use $S_0$ for the set $\Ss{x \in \{0,1\}^\omega}{\text{the set } 1x^{-1} \text{ is finite}}$ and use $S_1$ for the set $\Ss{x \in \{0,1\}^\omega}{\text{the set } 0x^{-1} \text{ is finite}}$. All elements of $V_2$ stabilise both $S_0$ and $S_1$ setwise. The non-identity element of $\It{\{0,1\}}\alpha_{\{0,1\}}$ maps $S_0$ bijectively to $S_1$ and maps $S_1$ bijectively to $S_0$.

%From the last paragraph it follows that all elements of $G$ either stabilise both $S_1$ and $S_2$ setwise or switch them. From this it follows that there is a homomorphism from $G$ to the two element group sending elements which stabilise both $S_1$ and $S_2$ setwise to the identity and sending elements which switch $S_1$ and $S_2$ to the non-identity element. Therefore $G$ is not simple.
%\end{exmp}

%Proposition \ref{prop:Age} below is used in the proof of Theorem \ref{thm:sub2}. If $C$ is in $\Cc$ and $G$ is a subgroup of $\autc$ then we use $CG$ to denote $\Ss{C\gamma}{\gamma \in G}$, that is the orbit of $C$ in $\Cc$ under $G$.

\begin{defn} \label{def:Age}
Let $G \leq \autc$ be approximately full and vigorous. For $B$ in $\Cc$ we will use $\Oo{B}{G}$ to denote the set $\Ss{B\gamma}{\gamma \in G}$. We will use $\Ooo{G}$ to denote the set $\Ss{\Oo{C}{G}}{C \in\Cc}$. For $U, V\in \Cc$ define $\Oo{U}{G}+\Oo{V}{G}\seteq \Oo{U\mu\sqcup V\nu}{G}$ where $\mu, \nu\in G$ are so that $U\mu\cap V\nu=\varnothing$ and $U\mu\sqcup V\nu\neq \C$.
  In the next proposition we prove this operation is well defined.
\end{defn}

\begin{prop} \label{prop:Age}
Let $G \leq \autc$ be approximately full and vigorous. The operation $+$ defined in \ref{def:Age} is well defined and when equipped with this operation $\Ooo{G}$ is a commutative group.
\end{prop}
\begin{proof}
 We first show that the operation $+$ is well defined.
 
 Let $U, V\in \Cc$, we will show $\Oo{U\mu_1 \sqcup V\nu_1}{G} = \Oo{U\mu_2 \sqcup V\nu_2}{G}$ for any $\mu_1,\mu_2,\nu_1,\nu_2\in G$ with $U\mu_i\cap V\nu_i=\varnothing$ and $U\mu_i\sqcup V\nu_i\neq \C$ for $i\in \{1,2\}$. 
 
 Let $\mu_1$, $\mu_2$, $\nu_1$, $\nu_2\in G$ so that $U\mu_i\cap V\nu_i=\varnothing$ and $U\mu_i\sqcup V\nu_i\neq \C$ for $i\in \{1,2\}$.  
 
 We first prove a simplified case where the four sets $U\mu_1$, $U\mu_2$, $V\nu_1$, and $V\nu_2$ are pairwise disjoint and where the union of these four sets is not all of Cantor space.  
 
 Set $C\seteq \C\backslash(U\mu_1\sqcup U\mu_2\sqcup V\nu_1\sqcup V\nu_2)$ and consider $\mathcal{D}\seteq \{U\mu_1,U\mu_2,V\nu_1,V\nu_2, C\}$, a partition of $\C$ into five proper clopen sets.
 
 Consider $\gamma_1, \gamma_2, \gamma_3, \gamma_4, \gamma_5\in G$ taken as $\gamma_1=\mu_1^{-1}\mu_2$, $\gamma_2=\nu_1^{-1}\nu_2$, $\gamma_3=\gamma_1^{-1}$, $\gamma_4=\gamma_2^{-1}$ and $\gamma_5=1_G$, the identity of $G$.  As $G$ is approximately full, there is $\gamma\in G$ which agrees with $\gamma_1$ over $U\mu_1$, with $\gamma_2$ over $V\nu_1$, with $\gamma_3$ over $U\mu_2$, and with $\gamma_4$ over $V\nu_2$.  In particular, $(U\mu_1\sqcup V\nu_1)\gamma=(U\mu_2\sqcup V\nu_2)$ and therefore  $\Oo{U\mu_1\sqcup V\nu_1}{G}=\Oo{U\mu_2\sqcup V\nu_2}{G}$ in this case.
 
 Now let us drop the two assumptions that the four sets $U\mu_1$, $U\mu_2$, $V\nu_1$, and $V\nu_2$ are pairwise disjoint and that the union of these four sets is not all of Cantor space. Set $E\seteq\C\backslash(U\mu_1\sqcup V\nu_1)$ and $F\seteq
\C\backslash(U\mu_2\sqcup V\nu_2)$.  Choose $E'\subseteq E$ and $F'\subseteq F$ proper clopen subsets of $E$ and $F$ respectively so that the following properties hold:
\begin{enumerate}
\item $E\backslash E'\neq \varnothing$,
\item $F\backslash F'\neq \varnothing$,  
\item $E'\cap F' = \varnothing$, and
\item $E'\sqcup F'\neq \C$
\end{enumerate}
observing we can find these sets even if $E\cap F\neq \varnothing$.

Now, as $G$ is vigorous, there are group elements $\theta_1, \theta_2\in G$ so that $(U\mu_1\sqcup V\nu_1)\theta_1\subseteq E'$ while $(U\mu_2\sqcup V\nu_2)\theta_2\subseteq F'$.

By the definition of the sets $\Oo{A}{G}$ for $A\in \Cc$, we see that $\Oo{U\mu_1\sqcup V\nu_1}{G}=\Oo{(U\mu_1\sqcup V\nu_1)\theta_1}{G}$ and $\Oo{U\mu_2\sqcup V\nu_2}{G}=\Oo{(U\mu_2\sqcup V\nu_2)\theta_2}{G}$.  However, we have that the four sets $U\mu_1\theta_1, V\nu_1\theta_1, U\mu_2\theta_2, V\nu_2\theta_2$ are pairwise disjoint and have union not all of $\C$ (as they are contained in $E'\sqcup F'$).  Thus it follows from our result in the simplified case that $\Oo{U\mu_1\theta_1\sqcup V\nu_1\theta_1}{G}=\Oo{U\mu_2\theta_2\sqcup V\nu_2\theta_2}{G}$.  In particular, $+$ is well defined.

Now we show that the operation $+$ is associative. Let $X$, $Y$, $Z\in \Cc$. We will show that $(\Oo{X}{G} + \Oo{Y}{G}) + \Oo{Z}{G} = \Oo{X}{G} + (\Oo{Y}{G} + \Oo{Z}{G})$.  As $G$ is vigorous we may find  $\alpha$ and $\beta$ in $G$ so that $X\alpha$, $Y\beta$ and $Z$ are disjoint and where $X\alpha\cup Y\beta\cup Z\neq \C$. Now
\begin{align*}
(\Oo{X}{G} + \Oo{Y}{G}) + \Oo{Z}{G} & = (\Oo{X\alpha}{G} + \Oo{Y\beta}{G}) + \Oo{Z}{G} \\
& = (\Oo{X\alpha \sqcup Y\beta}{G}) + \Oo{Z}{G} \\
& = \Oo{X\alpha \sqcup Y\beta \sqcup Z}{G} \\
& = (\Oo{X\alpha}{G} + \Oo{Y\beta \sqcup Z}{G}) \\
& = \Oo{X\alpha}{G} + (\Oo{Y\beta}{G} + \Oo{Z}{G}) \\
& = \Oo{X}{G} + (\Oo{Y}{G} + \Oo{Z}{G})
\end{align*}
as desired.

Since the unions are commutative it follows that the operation $+$ is commutative.  It remains to show that $\Ooo{G}$ under the operation of $+$ forms a group. 

Let $A$ and $C\in \Cc$.  It is sufficient to find $B \in \Cc$ such that $\Oo{A}{G} + \Oo{B}{G} = \Oo{C}{G}$.

Since $G$ is vigorous we may find $\beta$ in $G$ such that $A\beta \Sne C$. Let $B$ be equal to $C \Sm A\beta$. Now $\Oo{A}{G} + \Oo{B}{G} = \Oo{A\beta \sqcup B}{G} = \Oo{C}{G}$ as desired. 
\end{proof}
Below, we use $\Oo{\varnothing}{G}$ to denote the set $\Ss{A\gamma \Sm A}{A \in \Cc \textrm{ and } \gamma \in G \textrm{ and } A\subsetneq A\gamma}$ and use $\Oo{\C}{G}$ to denote the set $\Ss{\C \Sm A}{A \in \Oo{\varnothing}{G}}$.  We have the following lemma.

\begin{lem}\label{lem:gpStructure}
Let $G \leq \autc$ be approximately full and vigorous.  We have the following.
\begin{enumerate}
\item \label{pt:identity} The set $\Oo{\varnothing}{G}$ is an element of $\Ooo{G}$ and for all $A\in \Cc$, we have $\Oo{A}{G}+\Oo{\varnothing}{G}=\Oo{A}{G}$. 
\item \label{pt:invTot} Let $A\in \Cc$,
\begin{enumerate}
\item \label{pt:inverse}if $\alpha,\beta\in G$ so that $A\subsetneq A\alpha$ and $A\beta \subsetneq A\alpha\backslash A$, then $\Oo{A\alpha\backslash (A\sqcup A\beta)}{G}$ is the inverse of $\Oo{A}{G}$, and 
\item \label{pt:totality}$\Oo{A}{G}+\Oo{\C\backslash A}{G}=\Oo{\C}{G}$.
\end{enumerate}
\end{enumerate}
\end{lem}

\begin{proof}
Let $A, B\in \Cc$ and $\alpha, \beta\in G$ so that $A\subsetneq A\alpha$ and $B\subsetneq B\beta$.

It is the case that $\Oo{A\alpha \backslash A}{G}+\Oo{A}{G}= \Oo{A\alpha}{G}=\Oo{A}{G}$, so $\Oo{A\alpha \backslash A}{G}$ is the identity for the group $\Ooo{G}$.  Similarly, $\Oo{B\beta \backslash B}{G}$ is also the identity of $\Ooo{G}$, so we have $\Oo{A\alpha \backslash A}{G}=\Oo{B\beta \backslash B}{G}$.  In particular there is $\gamma\in G$ so that $(A\alpha \backslash A)\gamma = (B\beta\backslash B)$.  From this it follows that $\Oo{\varnothing}{G}\subseteq (A\alpha\backslash A)G=\Oo{A\alpha \backslash A}{G}$.

Observing that for any $C\in\Cc$ and $\delta,\rho\in G$ with $C\subsetneq C\rho$ we have $(C\rho\backslash C)\cdot \delta = (C\delta)\rho^\delta\backslash C\delta$  we see that $\Oo{\varnothing}{G}\supseteq(A\alpha\backslash A)G=\Oo{A\alpha\backslash A}{G}$.  It now follows that $\Oo{\varnothing}{G}=\Oo{A\alpha\backslash A}{G}$, and so $\Oo{\varnothing}{G}$ is an element of $\Ooo{G}$ and in particular it is the identity under $+$.

In passing we note that the argument above shows that the set $\Oo{\C}{G}$ is a group element, as the orbit of the complement of a proper clopen set is the orbit of a proper clopen set.

For point (\ref{pt:inverse}) we recall from the last paragraph of the proof of Proposition \ref{prop:Age} that if $A,C\in \Cc$ then choosing any $\beta$ so that $A\beta\subsetneq C$ we obtain the formula $\Oo{A}{G}+\Oo{C\backslash A\beta}{G}=\Oo{C}{G}$.  Therefore, if we consider some $\alpha\in G$ so that $A\subsetneq A\alpha$ then  $(A\alpha\backslash A)\in \Cc$ and $\Oo{\varnothing}{G} = \Oo{A\alpha\backslash A}{G}$.  Thus, setting $C\seteq A\alpha\backslash A$ in our formula and choosing $\beta\in G$ so that $A\beta\subsetneq C$ then we obtain the desired result for the inverse of $\Oo{A}{G}$.

Point (\ref{pt:totality}) follows easily as well.  Let $A\in\Cc$, and choose some $\gamma$ so that $A\gamma \subsetneq A$.  It then follows that $\Oo{A}{G}+\Oo{\C\backslash A}{G}= \Oo{A\gamma}{G} +\Oo{\C\backslash A}{G} = \Oo{A\gamma \sqcup \C\backslash A}{G}= \Oo{\C\backslash (A\backslash A\gamma)}{G}$.
But $\Oo{A\backslash A\gamma}{G}=\Oo{\varnothing}{G}$ so $\Oo{\C\backslash (A\backslash A\gamma)}{G}=\Oo{\C}{G}$ by the definition of $\Oo{\C}{G}$ and so we have our desired result.
\end{proof}
The reader may check the following.
{
\begin{rmk}
A vigorous and approximately full group \(G\) is dense in \(\autc\) if and only if \(\Ooo{G}\) is trivial.
\end{rmk}
}

\subsubsection{Four Lemmas towards Theorem \ref{sub2}}

The following lemma is a well known and useful purely algebraic statement about finitely generated subgroups of groups.

\begin{lem} \label{lem:res}
If $F \leq G$ are groups, $F$ is finitely generated, and $Y$ is a generating set for $G$, then there exists a finite subset $Y_0$ of $Y$ with $\An{Y_0} \geq F$. 
\end{lem}
\begin{proof}
This is left as an exercise for the reader.
%\textcolor{red}{Since $X \Se \An{Y}$ we may for each element of $X$ fix a product over $Y$ equal to that element of $X$. Since $X$ is finite and the set of elements of $Y$ used in any one of the products is finite, the set $Y_0$ of elements of $Y$ that are used in any of the products is finite. However $\An{Y_0} \supseteq X$ so $\An{Y_0} \geq \An{X} = F$ as desired.}
\end{proof}

We prove the next lemma to allow us to embed finite alternating groups in any group in $\K$ in a natural way.

We say an element of $x$ of $\Ooo{G}$ is  \emph{even} if there exists $y$ in $\Ooo{G}$ such that $x = y + y$.
\begin{lem} \label{lem:eag}
Let $G \leq \autc$ be vigorous and approximately full. Let $C$ be in $\Cc$ and let $\Gamma$ a finite subset of $G$ be such that if $\gamma_1, \gamma_2\in \Gamma$ are distinct then $C\gamma_1$ and $C\gamma_2$ are disjoint.

Let $\pi$ be a permutation of the set $\Gamma$. Let $\delta \in \autc$ be such that for each $\gamma \in \Gamma$ the homeomorphism $\delta$ agrees with $\gamma^{-1} (\gamma \cdot\pi)$ on $C\gamma$ and has support contained in $\bigcup\limits_{\gamma \in \Gamma} C\gamma$. If $\Oo{C}{G}$ is an even element of $\Ooo{G}$ or $\pi$ is an even permutation then $\delta$ is in $G$.
\end{lem}
\begin{proof}
First assume that $\pi$ is an even permutation.

We may assume $|\Gamma| \geq 3$ because the alternating group on $\Gamma$ is trivial if $|\Gamma|<3$.  If $|\Gamma|\geq 3$ then the embedded copy of the alternating group on $\Gamma$ is generated by its embedded $3$-cycles, so we only need address the case where $\Gamma = \{\alpha,\beta,\gamma\}$, a set of size $3$, and where we further assume that $\pi$ is the 3-cycle $(\alpha\ \beta\ \gamma)$.

First we assume $(C\alpha) \cup (C\beta) \cup (C\gamma) \subsetneq \C$.

Let $U$ and $V$ in $\Cc$ be disjoint so that $U$ and $V$ are also disjoint from $(C\alpha) \cup (C\beta) \cup (C\gamma)$. Let $\mu \in \pstab_G(\C \Sm ((C\alpha) \cup (C\beta) \cup V))$ be such that $\mu$ agrees with $\alpha^{-1}\beta$ on $C\alpha$ and $\beta^{-1}\alpha$ on $C\beta$. Let $\nu \in \pstab_G(\C \Sm ((C\beta) \cup (C\gamma) \cup U))$ be such that $\nu$ agrees with $\beta^{-1}\gamma$ on $C\beta$ and $\gamma^{-1}\beta$ on $C\gamma$.

By inspection the commutator $[\mu,\nu]$ has support contained in $(C\alpha) \cup (C\beta) \cup (C\gamma)$ and agrees with $\alpha^{-1}\beta$ on $C\alpha$ and agrees with $\beta^{-1}\gamma$ on $C\beta$ and agrees with $\gamma^{-1}\alpha$ on $C\gamma$ as desired.

Now assume $(C\alpha) \cup (C\beta) \cup (C\gamma) =\C$. Let $D \in \Cc$ be a proper subset of $C$.

Let $\sigma \in \autc$ be the homeomorphism of $\C$ with support equal to $(D\alpha) \cup (D\beta) \cup (D\gamma)$ which agrees with $\alpha^{-1}\beta$ on $D\alpha$ and agrees with $\beta^{-1}\gamma$ on $D\beta$ and agrees with $\gamma^{-1}\alpha$ on $D\gamma$. Let $\rho \in \autc$ be the homeomorphism of $\C$ with support equal to $((C \Sm D)\alpha) \cup ((C \Sm D)\beta) \cup ((C \Sm D)\gamma)$ which agrees with $\alpha^{-1}\beta$ on $(C \Sm D)\alpha$ and agrees with $\beta^{-1}\gamma$ on $(C \Sm D)\beta$ and agrees with $\gamma^{-1}\alpha$ on $(C \Sm D)\gamma$.

By the first case $\sigma$ and $\rho$ are both in $G$. The product $\sigma\rho$ is as desired and the case where $\pi$ is an even permutation is complete.

If $\Oo{C}{G}$ is even then we may find a non-empty proper clopen subset $C_0$ of $\C$ such that $\Oo{C_0}{G} + \Oo{C_0}{G} = \Oo{C}{G}$. We may assume that by the definition of $\Oo{C}{G}$ and of our ``$+$'' operator that $C_0$ is a non-empty proper subset of $C$, and in particular we may find $\tau$ in $G$ such that $C_0\tau = C \Sm C_0$.

Let $\Gamma_0$ be the set $\Gamma \cup \Ss{\tau \gamma}{\gamma \in \Gamma}$. Let $\pi_0$ be the permutation of $\Gamma_0$ such that $\gamma \cdot\pi_0 = \gamma \cdot\pi$ and $\tau\gamma \cdot\pi_0 = \tau(\gamma \cdot \pi)$.

Now the sets $\{C_0\gamma_0\}_{\gamma_0 \in \Gamma_0}$ are pairwise disjoint and $\pi_0$ is an even permutation of $\Gamma_0$. By applying the case for even permutations we find the desired homeomorphism.
\end{proof}

\begin{lem} \label{lem:part}
Let $G \leq \autc$ be approximately full and vigorous. Let $T$ be a non-empty finite set and let $f$ be a map from $T$ to $\Ooo{G}$. There exists a partition $\{C_t \in (t)f\}_{t \in T}$ of $\C$ exactly if $\sum\limits_{t \in T} ((t)f) = \Oo{\C}{G}$.
\end{lem}
\begin{proof}
It follows immediately from Lemma \ref{lem:gpStructure} that if such a partition exists then $\sum\limits_{t \in T} ((t)f) = \Oo{\C}{G}$.

Given such a set $T$ we will construct such a partition. First put a total order on $T$. Let $l$ be the greatest element of $T$. If $C_t$ has been assigned for each $t < q$ for some $q \in T \Sm\{l\}$ and $C_q$ has not been assigned then let $C_q$ be an element of $(q)f$ and a non-empty proper subset $\C \Sm \bigsqcup\limits_{t < q}C_t$. If $C_t$ has been defined for each $t \in T \Sm \{l\}$ then let $C_l$ be the set $\C \Sm \bigsqcup\limits_{t \in T \Sm \{l\}} C_t$. The set $C_l$ is in $(q)f$ by  Lemma \ref{lem:gpStructure}.
\end{proof}

Lemmas \ref{lem:eag} and \ref{lem:part} combine powerfully to construct bespoke (torsion) elements in $G$.  The first of these roughly says we can find elements of $G$ which induce ``even permutations'' of a partition of $\C$ into proper clopen sets by focusing on elements of the partition with the same homology type in $\Ooo{G}$, while the second informs us of the existence of partitions where the parts satisfy algebraic conditions measured by $\Ooo{G}$.  We will see a powerful use of this synergy in the proof of Theorem \ref{thm:sub2}.

The final lemma is a technical lemma which we apply in the proof of Theorem \ref{thm:sub2}.  In this lemma we use a closed real interval with integer bounds to represent the integers between and including those bounds.  E.g., $[3,6] = \{3,4,5,6\}$ in the notation below.
 
 %if $A$ is a set of natural numbers and $b$ is a natural number then $A - b$ is intended to mean the set $\Ss{a-b}{a \in A}$.

\begin{lem} \label{lem:num}
For each positive integer $j$ and $i \in [1, j]$ the set
%\[([1,l] \cup \Ss{i(l+1)}{i \in [1,l]})-i \cap ([1,l] \cup \Ss{i(l+1)}{i \in [1,l]})-i(l+1)\]
%is equal to $\{0\}$.
\[\Ss{xu - i}{x \in [1,j], u \in \{1,j+1\}} \cap \Ss{yv - i(j+1)}{y \in [1,j], v \in \{1,j+1\}}\]
is equal to $\{0\}$.
\end{lem}
\begin{proof}
Note that \(0\) is an element of this set. Let \(z\) be an arbitrary element of this intersection. We show that \(z=0\).

Suppose that $z=xu-i=yv-i(j+1)$ where $x,y \in [1,j]$, $u,v\in \{1,j+1\}$. 
It follows that $x-i\equiv y-i$ mod \(j\), so $x=y$ and $u<v$. It follows that $u=1$ and $v=j+1$. Therefore $z=x-i=x(j+1)-i(j+1)$ and so $ij=xj$ and \(x=i\). Finally $z=i-i=0$ as required.
\end{proof}

\subbtwo
\begin{proof}
Let $C \in \Cc$ be such that $\Oo{C}{G} = \Oo{\C}{G}$.
Let $D \in \Cc$ be a proper subset of $\C \Sm C$ such that $\Oo{D}{G} = -2\Oo{C}{G}$.

By Proposition \ref{prop:psk} the group $\pstab_G(\C \Sm (C \sqcup D))$ is in $\Kkf{C \sqcup D}$.  Now set 
$$P \seteq \Ss{\alpha \in \pstab_G(\C \Sm (C \sqcup D))}{\alpha \textrm{ is of order } n}.$$
By Lemma \ref{lem:eag} arbitrarily large finite alternating groups embed in $\pstab_G(\C \Sm (C \sqcup D))$ therefore $P$ must be non-empty.

Since $\pstab_G(\C \Sm (C \sqcup D))$ is simple and the set $P$ is closed under conjugation by elements of $\pstab_G(\C \Sm (C \sqcup D))$ it follows that $P$ generates $\pstab_G(\C \Sm (C \sqcup D))$.
Let $Q \seteq \Ss{[\alpha,\beta]}{\alpha,\beta \in P}$.
The group $\pstab_G(\C \Sm (C \sqcup D))$ is infinite and simple and therefore not commutative. Consequently there must be non-identity elements of $Q$.

Since $\pstab_G(\C \Sm (C \sqcup D))$ is simple and the set $Q$ is closed under conjugation the set $Q$ must generate $\pstab_G(\C \Sm (C \sqcup D))$.

By Lemma \ref{lem:res} there exists a $j$ and elements $\{\psi_i\}_{0\leq i < j} \Se \pstab_G(\C \Sm (C \sqcup D))$ and $\{\omega_i\}_{0 \leq i < j} \Se \pstab_G(\C \Sm (C \sqcup D))$ of order $n$ such that the set $\{[\psi_i,\omega_i]\}_{0 \leq i < j}$ generates $\pstab_G(\C \Sm (C \sqcup D))$.

Let $x \in \N$ be at least $2j(j+1)$.

By Lemma \ref{lem:part} we may choose a partition $P$ of $\C$ into proper clopen sets so that $|P|=3x+1$, and where $P$ itself admits a partition as $X\sqcup Y$ so that
\begin{enumerate}
\item $X\subseteq [C]_G$, \item $Y\subseteq [D]_G$, 
\item $|X| = 2x+1$, and
\item $|Y|=x.$
\end{enumerate}  (For example let \(T= \{1, \ldots, 2x+1\}\sqcup \{-1, \ldots, -x\}\) and let \(f:T\to \Ooo{G}\) be the function  mapping the positive values to \(\Oo{C}{G}\in \Ooo{G}\) and the negative values to \(\Oo{D}{G}\in \Ooo{G}\).) We may further assume that $C \in X$ and $D \in Y$.

By Lemma \ref{lem:eag} there exists $\tau \in \pstab_G(\C \Sm \bigsqcup X)$ of order \(2x+1\) respecting $X$ such that the induced action of $\tau$ on $X$ has only one orbit (i.e., a $2x+1$-cycle over the elements of $X$).
Since $\Oo{D}{G}$ is even we may apply Lemma \ref{lem:eag} to find $\pi \in \pstab_G(\C \Sm \bigsqcup Y)$ of order \(x\) respecting $Y$ such that the induced action of $\pi$ on $Y$ has only one orbit.

Let $\sigma \seteq \pi\tau$. The order of $\sigma$ is the least common multiple of $2x+1$ and $x$, which is $2x^2+x$ as $2x+1$ and $x$ are coprime. 
 In particular, the order of $\sigma$ is finite.

Let $R$ be the graph with vertex set $\Ss{(C \sqcup D)\sigma^k}{k \in \Z}$ and an edge between two vertices if they properly intersect. Note that $R$ is a finite graph. Since $2x+1$ and $x$ are coprime the graph $R$ is connected. Since $R$ is connected and $\bigcup\limits_{k \in \Z}(C \sqcup D)\sigma^k = \mathfrak{C}$ it follows from repetitive application of Lemma \ref{lem:SoP} that the set $\{\sigma\} \sqcup \pstab_G(\C \Sm (C \sqcup D))$ generates $G$.

Therefore it is sufficient to find $\zeta \in G$ of order $n$ such that $\An{\sigma,\zeta} \supseteq \{[\psi_i,\omega_i]\}_{0 \leq i < j}$.
By the choice of $x$ it follows that the supports of the elements in the set
\[\{\psi_i^{\left(\sigma^i\right)}\}_{1 \leq i \leq j} \sqcup \{\omega_i^{\left(\sigma^{i(j+1)}\right)}\}_{1 \leq i \leq j}\]
 are disjoint so the product of all of them is well defined even without an order specified. Let $\zeta$ be equal to this product, which has order $n$.

We will show for each $1 \leq k \leq j$ the commutator $\left[\zeta^{\left(\sigma^{-k}\right)},\zeta^{\left(\sigma^{-k(j+1)}\right)}\right]$ is equal to the commutator $[\psi_k,\omega_k]$. This is sufficient to complete the proof.

Let $1 \leq k \leq j$ be given. Note
\begin{align*}
\suppp{\zeta}
& = \bigsqcup\limits_{1\leq i \leq j} \suppp{\psi_i^{\left(\sigma^i\right)}}
\sqcup
\bigsqcup\limits_{1\leq i \leq j}\suppp{\omega_i^{\left(\sigma^{i(j+1)}\right)}} \\
& = \bigsqcup\limits_{1\leq i \leq j} (C \sqcup D) \sigma^i \sqcup (C \sqcup D) \sigma^{i(j+1)}.
\end{align*}
Now $\suppp{\zeta^{\left(\sigma^{-k}\right)}} \cap \suppp{\zeta^{\left(\sigma^{-k(j+1)}\right)}}$ is equal to the intersection of
\[\left(\bigsqcup\limits_{1\leq i \leq j} (C \sqcup D) \sigma^i \sqcup (C \sqcup D) \sigma^{i(j+1)}\right)^{\left(\sigma^{-k}\right)}\]
and
\[\left(\bigsqcup\limits_{1\leq i \leq j} (C \sqcup D) \sigma^i \sqcup (C \sqcup D) \sigma^{i(j+1)}\right)^{\left(\sigma^{-k(j+1)}\right)}\]
which by Lemma \ref{lem:num} is equal to $C \sqcup D$. Here note we are relying on the fact that $x$ is at least $2j(j+1)$. It only remains to note that $\left[\zeta^{\left(\sigma^{-k}\right)},\zeta^{\left(\sigma^{-k(j+1)}\right)}\right]$ agrees with $[\psi_k,\omega_k]$ on $C \sqcup D$.
\end{proof}

\begin{cor}
If $G$ is in $\Kf$ then $G$ is $2$-generated as a semigroup.
\end{cor}
\begin{proof}
This follows immediately from Theorem \ref{thm:sub2}.
\end{proof}

\begin{exmp}
The group $V_2$ is a simple \vig{} and finitely generated subgroup of $\autc$ so Theorem \ref{thm:sub2} applies and $V_2$ is $2$-generated. 
\end{exmp}
%\begin{proof}
%The group $V_2$ is \vig{} as mentioned in subsection %\ref{CS}. That $V_2$ is simple and finitely generated is %proved in \cite{fpog} and earlier by Thompson himself in %unpublished notes \cite{Thompson}.
%\end{proof}

We note that Mason shows in \cite{MASON} that $V_2$ is $2$-generated, while later, 
Bleak and Quick in \cite{BleakQuick} find elements  they denote as
$u$, $v\in V_2$ so that $V_2=\langle u,v\rangle$ where $|u|=6$ and 
$|v|=3$.  It is immediate from Theorem \ref{thm:sub2} that $V_2$ admits two torsion 
elements $\sigma$, $\zeta$ with $V_2=\langle \sigma,\zeta\rangle$ where $|\zeta|=2$.  Mark Sapir in personal communication 
with the first author asked whether $V_2$ is a quotient of 
PSL$_2(\Z)\cong C_3*C_2$ (so in particular, can $\sigma$ and $\zeta$ as above be found with $|\sigma|=3$).
\begin{exmp}
 Similar to the above, the simple R\"over-Nekrashevych groups of \cite{Nekra04} are all two generated, being finite-index (yet still vigorous) simple subgroups of specific finitely generated over-groups of the groups $V_n$.  See \cite{Roever,Nekra04,SkipperWitzelZaremsky} for more on these groups.   
\end{exmp}

\begin{exmp}
In our final example, after \cite{HennigMatucci} it was known that $nV$ admitted a presentation with $2n+4$ generators.  Martyn Quick shows in \cite{Quick} that for all $n$, the group $nV$ admits a finite presentation with two generators (matching our result here, but also providing explicitely the generators).
\end{exmp}

\section{Conclusion\label{sec:conc}}%%%%%%%%%%%%%%%%%%%%%%%%%%%%%%%%%%%%%%%%%%%%%%%%%%%%%%%%%%%%%%%%%%%%%%%%%%%%%%%%%%%%%%%%%%%%%%%%%%%%%%%%%%%%%%%%%%%%%%%%%%%%%%%%%%%%%%%%%%%%%

\begin{que} \label{fps}
Does there exist a finitely presented simple group that is not $2$-generated?
\end{que}

Question \ref{fps} is well-known and has been partially answered by the Classification of Finite Simple Groups from which it follows that if there are finitely presented simple groups which are not $2$-generated then they must be infinite. There are examples of simple groups which are finitely generated but not $2$-generated as shown by Guba in \cite{GUBA} and there are certainly examples of finitely presented non-$2$-generated groups and examples of finitely presented simple groups so there are no obstructions to any proper subset of the demands of the question.

Epstein in \cite{Epstein} provides three axioms under which a group of homeomorphisms of a space will be simple (or at least, will have a simple commutator subgroup).  One of these is that there is a countable basis of open sets upon which the group acts transitively, another subsumes the property that the group is generated by elements which are with support contained in these basic open sets.  These properties seem not far from the properties we have used in this paper to find $2$ element generating sets for finitely generated groups of homeomorphisms of the Cantor set.

For this reason, it may be an interesting project for the future to try to adapt the argument of the main theorem of this paper to the context of Epstein's Axioms, and thus investigate the following question.

\begin{que} If $G$ is a finitely generated simple group of homeomorphisms which satisfies Epstein's Axioms, must $G$ actually be two-generated?
\end{que}

%We should change the Proposition below into either of the (equivalent) Propositions:
%The family of isomorphism classes of subgroups of groups in $\K$ is equal to the family of isomorphism classes of countable groups.
%If $G$ is a countable group then there exists a group $H$ in $\Kf$ and an embedding of $G$ into $H$.
The following result is well known.  Note that the proof we give does not result in a vigorous group. 
\begin{prop} \label{prop:acige}
The group $Sym(\mathbb{N})$ embeds in $\autc$.
\end{prop}
\begin{proof}
For this proof we will think of elements of $\C$ as infinite words over $\{0,1\}$. Let $g$ be in $Sym(\mathbb{N})$ and let $x$ be in $\C$. If $x$ contains only the letter $0$ then let $g$ fix $x$. If the first $1$ of $x$ is the $(n+1)$th digit of $x$ then remove the first $n$ digits (all of which will be $0$) and replace them with a word of length $ng$ containing only the letter $0$ and define $xg$ to be the result. Extend this principle to define the action of $Sym(\mathbb{N})$ on the Cantor set. By inspection $Sym(\mathbb{N})$ now acts faithfully as desired.
\end{proof}

The following corollary is immediate from Cayley's Theorem.
\begin{cor}\label{cor:countableembeds}
All countably infinite groups embed in $\autc$.
\end{cor}
 From the above, one might hope to achieve vigorous realisations of all finitely presented simple groups and thus show that all finitely presented simple groups are two generated.  However, this fails as follows.
 
\begin{prop}
Every vigorous simple group has torsion elements.
\end{prop} 
\begin{proof}

A vigorous simple group $G$, by Theorem \ref{thm:fes}, is the commutator subgroup of its own full group.  Any such full group would admit torsion elements of small support, and by commutating one of these with an element moving that support fully off of itself we obtain a non-trivial torsion element in $G$.
\end{proof}

In particular, the Burger-Mozes torsion-free finitely presented simple groups \cite{BurgerMozes} admit no vigorous realisation in $\autc$. 

We remain interested in generalisations of the work in this paper from Cantor spaces to other spaces of interest (such as manifolds).  In such context, a natural generalisation of our concept of vigorous could be as follows.

\begin{defn}
Let $X$ be a topological space.  Let $H\leq \homeo(X)$.  We will say that a subset $G$ of $H$ is \emph{\vig{} with respect to $H$} if and only if for all $A,B,C$ open non-empty subsets of $X$ so that $\overline{B}$ and $\overline{C}$ are proper subsets of $A$  where there exists $h$ in $H$ with $\suppp{h} \Se A$ so that $\overline{B}h \Se C$, there is $g\in G$ with $\suppp{g}\Se A$ so that $\overline{B}g\Se C$.
\end{defn}

If $X=\C$ and $H=\homeo(\C)$ then a group $G$ being vigorous with respect to $H$ implies $G$ is vigorous. The generality of the definition above allows us to suppose that by restricting attention to less capable groups than $\homeo(X)$ we can still obtain interesting generation results.  For example, if $X$ had a measure, and if $H$ were the measure preserving subgroup of $\homeo(X)$, then an obstacle to a proof approach analogous to that of Theorem \ref{sub2} is that we frequently map clopen sets into proper subsets of themselves. Such maps are unlikely to respect the ambient measure.

This leads us to the following question.
\begin{que}
For which space-group pairs $(X,H)$ is it the case that if $G\leq \homeo(X)$ is finitely generated, simple and vigorous with respect to $H$ then $G$ is two-generated?
\end{que}

The following two questions are in part motivated by similar considerations for a family $\mathscr{F}$ of groups inspired by the theory of universal sequences \cite{Galvin95,HJMP16} and described below. 

Firstly, let $\mathscr{L}$ be the family of those $2$-generated groups $H$ such that for any group $G$ in $\Kf$ there exists an epimorphism from $H$ to $G$.

Note that by Theorem~\ref{thm:sub2}, $\mathscr{L}$ contains \(C * \mathbb{Z}\) for all non-trivial cyclic groups \(C\).

\begin{que}
What more can be said about the family $\mathscr{L}$?% of those $2$-generated groups $H$ such that for any group $G$ in $\Kf$ there exists an epimorphism from $H$ to $G$.
\end{que}

Let \(F_{\mathbb{N}}\) denote the free group of counably infinite rank.
Let $\mathscr{F}$ be the family of those $2$-generated groups $G$ such that there exists a homomorphism $\phi:F_\N \to G$ such that for any homomorphism $\psi:F_\N \to \It{\N}$ there exists a homomorphism $\rho:G \to \It{\N}$ such that $\phi\rho = \psi$.

Clearly groups in $\mathscr{L}$ or $\mathscr{F}$ must be lawless, have exponential growth, and have uncountably many normal subgroups.

Galvin effectively shows in the proof of Theorem 4.3 of \cite{Galvin95} that for any $i\geq 3$ and even \(j\geq 4\), the group $C_i*C_j$ is in $\mathscr{F}$, which motivates the following question.

%Galvin effectively shows in \cite{Galvin95} that for any $i,j$ both at least $2$ and not both $2$ the group $C_i*C_j$ is in $\mathscr{F}$, which motivates the following question.

\begin{que}\label{que:2-3gen}
Are there elements of $\Kf$ which are not quotients of $C_2*C_3$?
\end{que}

The families $\mathscr{L}$ and $\mathscr{F}$ both arise as being compatible with finite generation in uncountably many ways corresponding to groups that are infinite versions of finite symmetric groups.

\hrulefill
\vspace {.1 in}
{\flushleft \small Collin Bleak\\
School of Mathematics and Statistics, \\
University of St Andrews,\\
St Andrews, Scotland, KY8 6ER,\\
(e-mail): cb211@st-andrews.ac.uk}

{\flushleft \small  Luke Elliott\\
School of Mathematics and Statistics,\\
University of St Andrews,\\
St Andrews, Scotland, KY8 6ER,\\ (e-mail): le27@st-andrews.ac.uk}

{\flushleft \small  James Hyde\\
310 Malott Hall,\\
Cornell University,\\
Ithaca, New York 14853,\\
(e-mail): jth263@cornell.edu}

\vspace {.1 in}

{\flushleft Competing interests: the authors declare none.}

\hrulefill

\bibliographystyle{amsplain}
\bibliography{f_autq}

\end{document}